\documentclass[10pt]{article}
\usepackage{amsmath} %many equation formats (split, align, multline); DeclareMathOperator; \text command; (also loads amsbsy for backwards compatibility)
\usepackage{amssymb} %many symbols not used often;(loads amsfonts-- fraktur, blackboard bold)
\usepackage{amsthm}  %does theorem environments
\usepackage{fullpage} %sets margins to be 1in or 1.5cm; other options exist

\usepackage{graphicx} %allows graphics; (loads graphics as well)
\usepackage{epstopdf} %converts EPS to PDF
\usepackage{lpic}  % for putting latex overlaid on an image

\usepackage{multicol} %to do lists in multiple columns
\usepackage{soul} %for tighter underlining; alternatively, can use \underline{\smash{xxxx}} which is maybe sometimes too tight

\setul{2pt}{.4pt}

\newtheorem{thm}{Theorem}[section]
\newtheorem{lemma}[thm]{Lemma}
\newtheorem{coro}[thm]{Corollary}
\newtheorem{claim}{Claim}[thm]

\newtheorem{prop}[thm]{Proposition}
\newtheorem{conj}[thm]{Conjecture}
\newtheorem{observation}[thm]{Observation}

\theoremstyle{definition}

\newtheorem{defn}{Definition}

\begin{document}

\title{Structure in sparse $k$-critical graphs}
\author{Ron Gould$^{1}$ \and Victor Larsen$^{2}$ \and Luke Postle$^{3}$}

\maketitle
\footnotetext[1]{Department of Mathematics and Computer Science, Emory University, Atlanta, GA 30322; {\tt rg@mathcs.emory.edu}.}
\footnotetext[2]{Department of Mathematics, University of Sioux Falls, Sioux Falls, SD, 57105; {\tt victor.larsen@usiouxfalls.edu}.}
\footnotetext[3]{Department of Combinatorics and Optimization, University of Waterloo, Waterloo, ON, Canada, N2L 3G1; {\tt lpostle@uwaterloo.ca}.}

\begin{abstract}
Recently, Kostochka and Yancey \cite{kostochkayancey2014} proved that a conjecture of Ore is asymptotically true by showing that every $k$-critical graph satisfies 
$|E(G)|\geq\left\lceil\left(\frac{k}{2}-\frac{1}{k-1}\right)|V(G)|-\frac{k(k-3)}{2(k-1)}\right\rceil.$
They also characterized \cite{KosYankOre} the class of graphs that attain this bound and showed that it is equivalent to the set of $k$-Ore graphs.
We show that for any $k\geq33$ there exists an $\varepsilon>0$ so that if $G$ is a $k$-critical graph, then $|E(G)|\geq\left(\frac{k}{2}-\frac{1}{k-1}+\varepsilon_k\right)|V(G)|-\frac{k(k-3)}{2(k-1)}-(k-1)\varepsilon T(G)$,
where $T(G)$ is a measure of the number of disjoint $K_{k-1}$ and $K_{k-2}$ subgraphs in $G$.
This also proves for $k\geq33$ the following conjecture of Postle \cite{PostleArXiv} regarding the asymptotic density:
For every $k\geq4$ there exists an $\varepsilon_k>0$ such that if $G$ is a $k$-critical $K_{k-2}$-free graph, then $|E(G)|\geq \left(\frac{k}{2}-\frac{1}{k-1}+\varepsilon_k\right)|V(G)|-\frac{k(k-3)}{2(k-1)}$.
As a corollary, our result shows that the number of disjoint $K_{k-2}$ subgraphs in a $k$-Ore graph scales linearly with the number of vertices and, further, that the same is true for graphs whose number of edges is close to Kostochka and Yancey's bound.
\end{abstract}

\section{Introduction}

Given a graph $G$ the \emph{chromatic number} of $G$, denoted $\chi(G)$, is the smallest integer $k$ such that there exists a mapping $\phi:V(G)\rightarrow\{1,\ldots,k\}$ where $\phi(u)\neq\phi(v)$ whenever $uv\in E(G)$. 
Such a mapping is called a \emph{proper $k$-coloring} of $G$.  
We say that $G$ is $k$-colorable if $G$ has a proper $k$-coloring.
There is an obvious connection between the number of edges in a graph and the graph's chromatic number.
Each edge is a restriction on the vertex labeling, and thus removing edges can lower the chromatic number of a graph.
Indeed, the chromatic number of $G-e$ is either $\chi(G)$ or $\chi(G)-1$.
It is natural to study the class of graphs which are as sparse as possible for a given chromatic number.

A graph $G$ is \emph{$k$-critical} if $\chi(G)=k$ and every proper subgraph is $(k-1)$-colorable.
Viewing $k$-critical graphs as minimal graphs with chromatic number $k$ leads to the question of how small such graphs can be.
Let $f_k(n)$ denote the minimum number of edges in a $k$-critical graph, Ore conjectured \cite{Ore} the following.

\begin{conj}[Ore 1967 \cite{Ore}]
If $k\geq4$, then
\[f_k(n+k-1)=f_k(n)+(k-1)\left(\frac{k}{2}-\frac{1}{k-1}\right).\]
\end{conj}

As $\delta(G)\geq k-1$ for any $k$-critical graph, it is clear that $f_k(n)\geq\frac{k-1}{2}n$.
Since Dirac's 1957 paper \cite{Dirac57}, there have been many improvements over the years to the bounds for $f_k(n)$ (\cite{Gallai}, \cite{KS99}, \cite{Krivelevich})
Recently, Kostochka and Yancey \cite{kostochkayancey2014} made an important breakthrough.

\begin{thm}[Kostochka and Yancey 2014 \cite{kostochkayancey2014}, Theorem 3]\label{thm:kCritLowerBound}
If $k\geq4$ and $G$ is $k$-critical, then 
\[|E(G)|\geq\left\lceil\left(\frac{k}{2}-\frac{1}{k-1}\right)|V(G)|-\frac{k(k-3)}{2(k-1)}\right\rceil.\]
\end{thm}

They also showed in \cite{KosYankOre} that the class of graphs which attain this bound are \emph{$k$-Ore graphs}, which are defined below. 
First, we begin with a definition used to construct $k$-Ore graphs.
\begin{defn}
An \emph{Ore composition} of two graphs $G_1$ and $G_2$ is a graph obtained by the following procedure:
(1) delete an edge $xy$ from $G_1$, (2) split some vertex $z$ of $G_2$ into two vertices $z_1$ and $z_2$ of positive degree, and (3) identify $x$ with $z_1$ and $y$ with $z_2$.
\end{defn}
Note that the Ore composition of two graphs is not unique, depending on which edge is deleted from $G_1$, which vertex $z$ of $G_2$ is split, and how the neighbors of $z$ are partitioned.
Indeed, even the order in which we list the graphs is important; when we say that $G$ is an Ore composition of $H$ and $F$ we mean that $G$ is one of the graphs obtained by an Ore composition where $H$ plays the role of $G_1$ (called the \emph{edge-side} of the composition) and $F$ plays the role of $G_2$ (called the \emph{split-side} of the composition). The identified vertices \ul{$xz_1$} and \ul{$yz_2$} are called the \emph{overlap vertices} of the composition. Further, we call the edge $xy$ from step (1) the \emph{replaced edge} of $H$ and call the vertex $z$ from step (2) the \emph{split vertex} of $F$.
\begin{defn}
A graph $G$ is a \emph{$k$-Ore graph} if it is in the smallest class of graphs containing $K_k$ which is closed under the Ore composition operation.
\end{defn}
Equivalently, this is the class of graphs obtainable by successive Ore compositions of either $K_k$ or other $k$-Ore graphs.

To prove Theorem \ref{thm:kCritLowerBound}, which shows that Ore's Conjecture is asymptotically true, Kostochka and Yancey established the following result on the density of a $k$-critical graph using a \emph{potential function},
\[\rho_{KY}(G):=(k-2)(k+1)|V(G)|-2(k-1)|E(G)|.\]
\begin{thm}[Kostochka and Yancey 2014 \cite{kostochkayancey2014}, Theorem 5]\label{thm:KY}
If $k\geq4$ and $G$ is $k$-critical, then $\rho_{KY}(G)\leq k(k-3).$
\end{thm}
In a later paper, they also showed the following.
\begin{thm}[Kostochka and Yancey 2016+ \cite{KosYankOre}, Theorem 6]\label{thm:KY2}
If $k\geq4$ and $G$ is $k$-critical, then $\rho_{KY}(G)=k(k-3)$ if and only if $G$ is a $k$-Ore graph.
\end{thm}

The $k$-Ore graphs are the graphs which attain the bound of Theorem \ref{thm:KY}, and hence it is natural to ask if an increase in edge density is possible when forbidding subgraphs which arise through Ore constructions.
In \cite{Luke4}, Postle shows an increase in asymptotic density for $4$-critical graphs when forbidding both $K_3$ and $C_4$ subgraphs. 
By a construction of Thomas and Walls \cite{TW}, it is not sufficient to forbid only $K_3$.
\begin{thm}[manuscript, \cite{Luke4}]
There exists $\varepsilon>0$ such that if $G$ is a $4$-critical graph of girth at least five, then 
\[|E(G)|\geq\left(\frac{5}{3}+\varepsilon\right)|V(G)|-\frac{2}{3}.\]
\end{thm}
For larger values of $k$, it is also not sufficient to forbid only $K_{k-1}$. 
This leads to the following conjecture.
\begin{conj}[\cite{PostleArXiv}]\label{conj:Luke}
For every $k\geq4$, there exists $\varepsilon_k>0$ such that if $G$ is a $k$-critical $K_{k-2}$-free graph, then
\[|E(G)|\geq\left(\frac{k}{2}-\frac{1}{k-1}+\varepsilon_k\right)|V(G)|-\frac{k(k-3)}{2(k-1)}.\]
\end{conj}
The conjecture has been proven for $k=5$ \cite{PostleArXiv} and $k=6$ \cite{4Postle}.
In this paper, we prove the conjecture for $k\geq33$.
The method of proof also gives information about the structure of $k$-Ore graphs; in particular, we also prove that there are linearly many vertex-disjoint $K_{k-2}$ subgraphs in $k$-Ore graphs.
In order to track vertex-disjoint $K_{k-2}$ subgraphs (including those inside $K_{k-1}$ subgraphs), we define the following graph parameter.

\begin{defn}\label{def:T}
When a graph $H$ is a disjoint union of $r$ copies of $K_{k-1}$ and $s$ copies of $K_{k-2}$ subgraphs, define $T(H):=2r+s$. Let $G$ be an arbitrary graph.  If $G$ is $K_{k-2}$-free, then $T(G)=0$.  Otherwise, define
\[T(G):=\max_{H\subseteq G}\{T(H)\mid H \text{ is a disjoint union of }K_{k-1}\text{ and }K_{k-2}\text{ components}\}.\]
\end{defn}

In a $k$-Ore graph, $T(G)$ can be shown to be lower-bounded by some constant times the number of vertices (Lemma \ref{thm:TboundonkOre}).
Using the subgraph-measuring parameter $T(G)$ we define the following modified potential function.

\begin{defn}\label{def:potential}
Let $\varepsilon=\frac{4}{k^3-2k^2+3k}$ and $\delta=(k-1)\varepsilon$. 
Given a graph $G$ define the \emph{$\varepsilon$-potential} to be 
\[\rho(G):=((k-2)(k+1)+\varepsilon)|V(G)|-2(k-1)|E(G)|-\delta T(G).\]
\noindent For a vertex subset $R\subseteq V(G)$, we define
\[\rho_{G}(R):=((k-2)(k+1)+\varepsilon)|R|-2(k-1)|E(G[R])|-\delta T(G[R]),\]
where $G[R]$ is the induced subgraph of $G$ on $R$.
\end{defn}

One can check that the construction of $\varepsilon$ guarantees that $\varepsilon\leq1$ for all $k\geq2$ (in particular, it is true for all values of $k$ covered in this paper). With this modified potential function in hand, we are now able to state the main result of this paper.

\begin{thm}\label{thm:Main1}
If $G$ is a $k$-critical graph that is not a $k$-Ore graph and $k\geq33$, then $\rho(G)\leq k(k-3)-2(k-1)$.
\end{thm}

We prove this using the potential method of Kostochka and Yancey; however, a limitation in the discharging method used restricts this result to the range where $k\geq33$. 
Because reductions used in our proof could possibly create $k$-Ore graphs as auxiliary graphs, it is important that we also establish bounds for the $\varepsilon$-potential of $k$-Ore graphs.
In Section \ref{sec:kOre}, we prove the following.

\begin{thm}\label{thm:Main2}
If $G$ is a $k$-critical graph that is a $k$-Ore graph and $k\geq4$, then
\begin{enumerate}
\item $\rho(G)=k(k-3)+k\varepsilon-2\delta$ if $G=K_k$, and
\item $\rho(G)\leq k(k-3)+|V(G)|\varepsilon-\left(2+\frac{|V(G)|-1}{k-1}\right)\delta$ if $G\neq K_k$.
\end{enumerate}
\end{thm}

Note that Theorem \ref{thm:Main2} is proven using similar methods for $k=5$ in \cite{PostleArXiv}. 
Removing the notation of $\varepsilon$-potential, Theorem \ref{thm:Main1} and \ref{thm:Main2} give the following corollaries.

\begin{coro}
If $k\geq33$ and $G$ is $k$-critical, then
\[|E(G)|\geq\left\lceil\frac{\left[(k-2)(k+1)+\varepsilon\right]|V(G)|-k(k-3)+2\delta-k\varepsilon-\delta T(G)}{2(k-1)}\right\rceil,\]
where $\varepsilon=\frac{4}{k^3-2k^2+3k}$, $\delta=(k-1)\varepsilon$, and $T(G)$ is the subgraph-measuring parameter from Definition \ref{def:T}.
\end{coro}

\begin{coro}\label{cor:k-2Free}
If $k\geq33$, then there exists some $\varepsilon_k>0$ such that if $G$ is $k$-critical and $K_{k-2}$-free, then
\[|E(G)|\geq\left(\frac{k}{2}-\frac{1}{k-1}+\varepsilon_k\right)|V(G)|-\frac{k(k-3)}{2(k-1)}.\]
\end{coro}

Corollary \ref{cor:k-2Free} confirms Conjecture \ref{conj:Luke} for $k\geq33$.
We note that the class of $K_s$-free $k$-critical graphs was also studied by Krivelevich \cite{Krivelevich}.

\subsection{Outline of Paper and Notation}
The paper is organized as follows.  In Section \ref{sec:Prelim}, we establish some values for $\varepsilon$-potential.
We also prove some results about list colorings which are used in Section \ref{sec:Discharging}.
In Section \ref{sec:kOre}, we prove Theorem \ref{thm:Main2} and also prove results about subgraphs in $k$-Ore graphs.  These results are needed for our approach to Theorem \ref{thm:Main1}.
Sections \ref{sec:Ext}--\ref{sec:Discharging} address Theorem \ref{thm:Main1}.
In Section \ref{sec:Ext}, we define an auxiliary graph constructed from a partial $(k-1)$-coloring of a graph, and prove lemmas about the $\varepsilon$-potential of said graph.
In Section \ref{sec:EdgeAdd}, we work towards an important lemma (Lemma \ref{lem:EdgeAddition}) which says that subgraphs in a minimal counterexample to Theorem \ref{thm:Main1} must be many edges away from being $k$-critical.
In Section \ref{sec:Cloning}, we prepare for discharging by proving results on the structure near vertices of low degree in a minimal counterexample to Theorem \ref{thm:Main1}.
In Section \ref{sec:Discharging}, we complete the proof of Theorem \ref{thm:Main1} using a discharging argument.

Throughout the paper, we make use of the following concepts and notation. 
Given a graph $G$, let $x,y$ be vertices of $G$ and $R$ be a proper vertex subset of $G$.
We use $G/$\ul{$xy$} to refer to the graph obtained from $G$ by identifying $x$ and $y$; that is $G/$\ul{$xy$} is obtained by deleting $x,y$ and adding a new vertex \ul{$xy$} which is adjacent to each vertex in $N_G(x)\cup N_G(y)$.
The \emph{boundary vertices} of $R$ (in $G$) is the set $\partial_GR:=\{u\in R \mid N_G(u)-R\neq\emptyset\}.$
The \emph{closed neighborhood} of $x$ is the set $N_G[x]:=N_G(x)\cup\{x\}$.

The \emph{maximum independent cover number} of $G$, denoted $\text{mic}(G)$, is the maximum of $\sum_{x\in I}\deg_G(x)$ over all independent sets $I\subseteq V(G)$.
For terms not defined here see \cite{west}.

%%%%%%%%%%%%%%%%%%%%%%%%%%%%%%%%%%%%%%%
%%%%%%%%%%%%%%%%%%%%%%%%%%%%%%%%%%%%%%%

\section{Preliminaries}\label{sec:Prelim}
When proving bounds on $\rho(G)$, it is important to know the $\varepsilon$-potential of complete graphs.
\begin{observation}\label{fct:CompletePotential}~
\begin{enumerate}
\item $\rho(K_k)=k^2-3k+k\varepsilon-2\delta$.
\item $\rho(K_1)=k^2-k-2+\varepsilon$.
\item $\rho(K_{k-1})=2k^2-6k+4+(k-1)\varepsilon-2\delta$.
\item For $1<\ell<k-1$, the $\varepsilon$-potential of $K_\ell$ is bounded by $\rho(K_\ell)\geq2k^2-4k-2+2\varepsilon$.
\end{enumerate}
\end{observation}

We now establish some edge bounds which will be needed for the final stage of discharging in Section \ref{sec:Discharging}.
Given a graph $G$ and vertex subsets $A,B$, we define 
$e(A,B)$ to be the number of edges from a vertex in $A$ to a vertex in $B$.
That is, let $e(A,B):=\sum_{a\in A}|N_{G[A\cup B]}(a)\cap B|$.
We use the following lemma due to Kierstead and Rabern.

\begin{lemma}[Kierstead and Rabern 2015 \cite{KiersteadRabern}, Main Lemma]\label{lem:Kernel}
Let $G$ be a nonempty graph and $f:V\rightarrow \mathbb{N}$ with $f(v)\leq \deg_G(v)+1$ for all $v\in V(G)$.
If there is an independent set $A\subseteq V(G)$ such that 
\begin{equation}\label{eq:Kernel}
e(A,V(G))\geq\sum_{v\in V(G)}\left[\deg_G(v)+1-f(v)\right],
\end{equation}
then $G$ has a nonempty induced subgraph $H$ that is $f_H$-choosable where $f_H(v):=f(v)+\deg_H(v)-\deg_G(v)$.
\end{lemma}

\begin{lemma}
Let $G$ be a $k$-critical graph with vertex subsets $A,B_0,B_1$ such that $A$ is independent, $\deg_G(a)=k-1$ for each $a\in A$, and $\deg_G(b)=k+i$ for each $b\in B_i$ where $i\in\{0,1\}$. Then $e(A,B_0\cup B_1)< |A|+2|B_0|+3|B_1|$.
\end{lemma}
\begin{proof}
Suppose that $G$ is a $k$-critical graph with vertex subsets $A,B_0,B_1$ such that $A$ is independent, $\deg_G(a)=k-1$ for each $a\in A$, and $\deg_G(b)=k+i$ for each $b\in B_i$ where $i\in\{0,1\}$. 
Let $B=B_0\cup B_1$.
Suppose to the contrary that $e(A,B_0\cup B_1)\ge |A|+2|B_0|+3|B_1|$ holds true.  

Let $f:A\cup B\rightarrow\mathbb{N}$ where $f(v)=\deg_{G[A\cup B]}(v)$ if $v\in A$ and $f(v)=\deg_{G[A\cup B]}(v)-1-i$ if $v\in B_i$. 
Then the right side of Equation \ref{eq:Kernel} becomes
\[\sum_{v\in A}1+\sum_{v\in B_0}2+\sum_{v\in B_1}3=|A|+2|B_0|+3|B_1|.\]
It follows from Lemma \ref{lem:Kernel} that $G[A\cup B]$, and thus $G$, has a nonempty induced subgraph $H$ that is $f_H$-choosable where $f_H(v):=f(v)+\deg_H(v)-\deg_{G[A\cup B]}(v)$.

Since $G$ is $k$-critical, there exists a $(k-1)$-coloring $\phi$ of $G-H$.  For each vertex $v\in V(H)\cap A$, there are at least $\deg_H(v)$ colors available and we see that $f_H(v)=\deg_H(v)$.  Similarly, for each $v\in V(H)\cap B$, there are at least $\deg_H(v)-1-i$ colors available and $f_H(v)=\deg_H(v)-1-i$.  Therefore, we can use $f_H$-choosability to extend $\phi$ to all of $G$, which is a contradiction.
\renewcommand{\qedsymbol}{$\blacksquare$}
\end{proof}

%%%%%%%%%%%%%%%%%%%%%%%%%%%%%%%%%%%%%%%
%%%%%%%%%%%%%%%%%%%%%%%%%%%%%%%%%%%%%%%
\section{k-Ore graphs}\label{sec:kOre}

Here, we build up results regarding $k$-Ore graphs, which will be needed in order to bound the $\varepsilon$-potential for the reductions of general $k$-critical graphs that we will be using in subsequent sections.

\begin{prop}\label{prp:kOreSequence}
Given a $k$-Ore graph $G$, there is a sequence of $k$-Ore graphs $G_1,G_2,\ldots,G_s$ where $G_1=K_k$, $G_s=G$, and for each $2\leq i\leq s$, the graph $G_i$ is an Ore composition of $G_{i-1}$ and a $k$-Ore graph.
\end{prop}
\begin{proof}
Let $G$ be a $k$-Ore graph. We will prove this by induction on $|V(G)|$. 
If $G$ is $K_k$ the result is trivial, so we may assume that $G$ is an Ore composition of two $k$-Ore graphs $G_1$ and $G_2$ with overlap vertices $\{x,y\}$.
By induction, there is a sequence $\mathcal{H}=H_1,H_2,...,H_r$ where $H_1=K_k$ and $H_r=G_1$ and each $H_{i}$ is an Ore composition of $H_{i-1}$ and a $k$-Ore graph. Then the desired sequence for $G$ is $\mathcal{H},G$.
\renewcommand{\qedsymbol}{$\blacksquare$}
\end{proof}
Using this proposition, one can picture each $k$-Ore graph as a copy of $K_k$ where some number of edges are replaced by split $k$-Ore graphs.  In fact, any $k$-Ore graph can be obtained by simultaneously replacing some edges of a $K_k$ with suitable split $k$-Ore graphs. Before examining $\varepsilon$-potential, we establish bounds on the subgraph-measuring parameter $T(G)$.

\begin{lemma}\label{lem:TwithkOre}
If $G$ is an Ore composition of $G_1$ and $G_2$, then $T(G)\geq T(G_1)+T(G_2)-2$. Moreover, if $G_1=K_k$ or $G_2=K_k$, then $T(G)\geq T(G_1)+T(G_2)-1$.
Further, if both $G_1$ and $G_2$ are $K_k$, then $T(G)=4$.
\end{lemma}
\begin{proof}
Suppose that $G$ is an Ore composition of $G_1$ and $G_2$. 
Let $e$ be the replaced edge of $G_1$ and $z$ be the split vertex of $G_2$. 
From the definition of an Ore composition, it follows that $T(G)\geq T(G_1-e)+T(G_2-\{z\})$, and hence $T(G)\geq T(G_1)+T(G_2)-2$.
If $G_1=K_k$, then $T(K_k-e)=2=T(K_k)$ so we get $T(G)\geq T(G_1)+T(G_2)-1$.
We obtain a similar result if $G_2=K_k$ as $T(K_k-z)=2$.
Further, if both $G_1$ and $G_2$ are $K_k$, then $T(G)=4$.
\renewcommand{\qedsymbol}{$\blacksquare$}
\end{proof}

Note that the conclusion of Lemma \ref{lem:TwithkOre} is symmetric.

\begin{lemma}\label{thm:TboundonkOre}
If $G$ is a $k$-Ore graph and $G\neq K_k$, then $T(G)\geq 2+\frac{|V(G)|-1}{k-1}$.
\end{lemma}
\begin{proof}

We proceed by induction on $|V(G)|$. Let $G$ be an Ore composition of two $k$-Ore graphs $G_1$ and $G_2$.
If both $G_1$ and $G_2$ are $K_k$, then $|V(G)|=2k-1$; in this case, $T(G)=4$ as desired.
Suppose instead that exactly one of $G_1,G_2$ is $K_k$.  Because the conclusion of Lemma \ref{lem:TwithkOre} is symmetric and any Ore composition of a graph with $K_k$ adds $k-1$ vertices, we may assume without loss of generality that $G_1=K_k$.
It follows that
\[T(G)\geq T(G_2)+1\geq\left(2+\frac{|V(G_2)|-1}{k-1}\right)+1=2+\frac{|V(G)|-1}{k-1},\]
as desired.
Finally, suppose that neither $G_1$ nor $G_2$ is $K_k$. Then as $|V(G)|=|V(G_1)|+|V(G_2)|-1$, it follows from Lemma \ref{lem:TwithkOre} and induction that
\[T(G)\geq\left(2+\frac{|V(G_1)|-1}{k-1}\right)+\left(2+\frac{|V(G_2)|-1}{k-1}\right)-2=2+\frac{|V(G)|-1}{k-1}.\]
\renewcommand{\qedsymbol}{$\blacksquare$}
\end{proof}

Using Lemma \ref{thm:TboundonkOre}, we now prove Theorem \ref{thm:Main2}.

\begin{proof}[Proof of Theorem \ref{thm:Main2}]
By the definition of $\varepsilon$-potential, it follows that $\rho(K_k)=k(k-3)+k\varepsilon-2\delta$.
Now suppose that $G$ is a $k$-Ore graph which is not $K_k$.
Then $G$ has $k+\ell(k-1)$ vertices and $\frac{(\ell+1)k(k-1)}{2}-\ell$ edges for some $\ell\geq1$. Using Lemma \ref{thm:TboundonkOre}, it is again a straightforward calculation to show that $\rho(G)\leq k(k-3)+|V(G)|\varepsilon-\left(2+\frac{|V(G)|-1}{k-1}\right)\delta.$
\renewcommand{\qedsymbol}{$\blacksquare$}
\end{proof}

It is essential for the proof of Theorem \ref{thm:Main1} to understand the behavior of certain subgraphs of $k$-Ore graphs.  Two useful subgraphs are defined below.

\begin{defn}
A subgraph $D\subseteq G$ is a \emph{diamond of $G$} if $D=K_k-uv$ and $\deg_G(x)=k-1$ for each $x\in V(D)-\{u,v\}$. The vertices $u$ and $v$ are the \emph{endpoints} of the diamond.
A subgraph $D'\subseteq G$ is an \emph{emerald of $G$} if $D'=K_{k-1}$ and $\deg_G(x)=k-1$ for each $x\in V(D')$.
\end{defn}

\begin{lemma}\label{lem:DiamondAwayFromPoint}
If $G$ is a $k$-Ore graph and $v\in V(G)$, then there exists a diamond or emerald of $G$ in $G-v$. 
\end{lemma}

\begin{proof}
We prove this by induction on $|V(G)|$.
Suppose that $G$ is a $k$-Ore graph and let $v\in V(G)$ be an arbitrary vertex.
If $G=K_k$, then $G-v$ is an emerald of $G$, as desired.
Therefore we may assume that $G$ is an Ore composition of two $k$-Ore graphs $G_1$ and $G_2$ with overlap vertices $\{a,b\}$.
We choose this composition to minimize $|V(G_1)|$, the order of the edge-side.
By induction, there is an emerald or diamond $D$ of $G_2$ not containing $\underline{ab}$. 
Hence, if $v\in V(G_1)$, then $D$ is as desired. 
So we may assume that $v\in V(G_2)-\{\underline{ab}\}$.

Now if $G_1=K_k$, then $G_1-ab$ is a diamond of $G$ not containing $v$ as desired.
Therefore, we may assume that $G_1$ is a composition of two $k$-Ore graphs $H_1$ and $H_2$ with overlap vertices $\{x,y\}$. 
By our choice of $G_1$ it follows that $ab\in E(H_1)$. 
Thus there is an emerald or diamond subgraph $D$ of $H_2$ not containing \ul{$xy$}. Note that $D$ is also an emerald or diamond of $G$ and $v\notin V(D)$, as desired.
\renewcommand{\qedsymbol}{$\blacksquare$}
\end{proof}

\begin{lemma}\label{lem:DiamondAwayFromGroup}
If $G$ is a $k$-Ore graph and $D=K_{k-1}$ is a subgraph of $G$, then either $G= K_k$ or there exists a diamond or emerald of $G$ disjoint from $D$.
\end{lemma}

\begin{proof}
We prove this by induction on $|V(G)|$.
Suppose that $G$ is a $k$-Ore graph and let $D=K_{k-1}$ be a subgraph of $G$.
When $G=K_k$, the lemma is trivial.
So we may assume that $G$ is an Ore composition of two $k$-Ore graphs $G_1$ and $G_2$ with overlap vertices $\{a,b\}$.
Choose this composition to minimize the order of the edge-side, $|V(G_1)|$.
As $\{a,b\}$ is an independent cutset in $G$, it follows that either $D\subseteq G_1-ab$ or $D\subseteq G_2$. 
If $D\subseteq G_1-ab$, then by Lemma \ref{lem:DiamondAwayFromPoint} there exists a diamond or emerald $D'$ of $G_2-\underline{ab}$ and $D'$ is disjoint from $D$ as desired.

Thus we may assume that $V(D)\subseteq V(G_2)\cup\{a,b\}$.
We examine two cases based on whether $V(D)$ contains any of the overlap vertices $\{a,b\}$ or not. 

Since $a$ is not adjacent to $b$ in $G_2$, they cannot both be in $D$.
So first, suppose that $|V(D)\cap\{a,b\}|=1$ and without loss of generality, we assume that $a\in V(D)$. 
If $G_2\neq K_k$, then by induction, there is a diamond or an emerald of $G_2$ disjoint from $D$ and this is also a diamond or an emerald of $G$, as desired.
Therefore we may assume that $G_2=K_k$ and thus $b$ has one neighbor on the split-side of $G$.
It follows that $\deg_{G_1}(b)=\deg_G(b)$.
By Lemma \ref{lem:DiamondAwayFromPoint} 
there is a diamond or emerald $D'$ of $G_1$ in $G_1-a$.
If $D'$ is a diamond, then $D'$ is also a diamond of $G$.
If $D'$ is an emerald, then because $\deg_{G_1}(b)=\deg_G(b)$, it follows that $D'$ is an emerald of $G$.
In either case, $D'\cap D=\emptyset$ as desired.

Second, suppose that $V(D)$ contains neither $a$ nor $b$.
If $G_1=K_k$, then $G_1-ab$ is a diamond that is disjoint from $D$. 
Otherwise, $G_1$ is a composition of two $k$-Ore graphs $H_1$ and $H_2$ with overlap vertices $\{x,y\}$. 
By our choice of $G_1$ it follows that $ab\in E(H_1)$. 
By Lemma \ref{lem:DiamondAwayFromPoint} 
there is a diamond or emerald $D'$ of $H_2-$\ul{$xy$}, which then contains no vertices of $D$.  Thus, $D'$ is also a diamond or emerald of $G$, as desired.
\renewcommand{\qedsymbol}{$\blacksquare$}
\end{proof}

%%%%%%%%%%%%%%%%%%%%%%%%%%%%%%%%%%%%%%%
%%%%%%%%%%%%%%%%%%%%%%%%%%%%%%%%%%%%%%%

\section{Critical Extensions}\label{sec:Ext}

We now turn towards proving the main result, Theorem \ref{thm:Main1}.
We do this
by discharging on a minimal counterexample; therefore we need to precisely define what makes a graph minimal.
\begin{defn}
A graph $H$ is \emph{smaller} than a graph $G$ if $|V(G)|>|V(H)|$ or, if $|V(G)|=|V(H)|$, then $H$ is smaller if either $|E(G)|>|E(H)|$ or if $|E(G)|=|E(H)|$ and $G$ has fewer pairs of vertices with the same closed neighborhood.
\end{defn}

Given a $k$-critical graph $G$, we have a particular method of examining what subgraphs exist in $G$. Note that if $R$ is a proper vertex subset of $G$, then we can properly $(k-1)$-color $G[R]$.  Such a coloring is used to create the following auxiliary graph.
\begin{defn}\label{def:colorReduce}
Given a $k$-critical graph $G$ and a proper $(k-1)$-coloring $\phi$ on a vertex subset $R$, we define the graph $G_{R,\phi}$ to be the graph obtained from $G$ by identifying all vertices in $\phi^{-1}(i)$ to a single vertex $x_i$ for $1\leq i\leq k-1$, adding the edge $x_ix_j$ for each $1\leq i<j\leq k-1$, and then deleting any parallel edges so that the new vertices form a complete subgraph with no parallel edges.
\end{defn}

Note that if $uv\in E(G)$ for $u\in R$ and $v\in V(G)-R$, then $vx_{\phi(u)}\in E(G_{R,\phi})$.
Further, we will always color $R$ with as few colors as possible, so then it follows that $G_{R,\phi}$ is a smaller graph than $G$, or possibly $G_{R,\phi}=G$ if $R$ is a clique.
One can observe that $G_{R,\phi}$ is not $(k-1)$-colorable; a proof of this is in \cite{kostochkayancey2014} (Claim 14).
Therefore, there is a $k$-critical subgraph $W\subseteq G_{R,\phi}$. Because $G$ is $k$-critical, $W$ must contain at least one vertex in $\{x_1,\ldots,x_{k-1}\}$.  The fact that $W$ is smaller than $G$ when $R$ is not a clique is used frequently in subsequent $\varepsilon$-potential calculations.
\begin{defn}
Given a graph $G_{R,\phi}$ obtained via Definition \ref{def:colorReduce} and a $k$-critical subgraph $W$, we define $R':=\left(R\cup V(W)\right)-X$ to be a \emph{$W$-critical extension of $R$} where $X:=V(W)\cap\{x_1,\ldots,x_{k-1}\}$ is called the \emph{core} of the $W$-critical extension. If $R'=V(G)$, then we say that $R'$ is a \emph{spanning} $W$-critical extension. Lastly, the $W$-critical extension $R'$ is \emph{complete} if 
\begin{equation}\label{eq:completeEq}
|E(G[R'])|=|E(G[R])|+|E(W)|-|E(K_{|X|})|.
\end{equation}
For a general $W$-critical extension $R'$, it is possible that the left side of Equation \ref{eq:completeEq} is larger. If we have $|E(G[R'])|=|E(G[R])|+|E(W)|-|E(K_{|X|})|+i$, then we say that the $W$-critical extension is \emph{$i$-incomplete}.
\end{defn}
Thus a $W$-critical extension is complete if the edges from $R$ to $V(W)$ in $G[R']$ correspond to the edges from $X$ to $V(W)-X$ in $W$, and incompleteness comes from three sources.  
First, edges from $R$ to $V(W)$ in $G[R']$ that are not needed in $W$ get counted on the left but never on the right.
Second, if $N_{R}(w)\cap ({\rm color\,\, } \ell)$ is larger than 1 for some $w\in V(W)-R$ and color $\ell$, then $|E(G[R'])|$ counts all of these edges but $|E(W)|$ counts at most one.  Third, if an edge $x_ix_j$ is not used in $W$, then it is not counted by $|E(W)|$ but is subtracted by $|E(K_{|X|})|$.

\begin{lemma}\label{lem:PotentialUnderExtension}
Suppose that $G$ is a $k$-critical graph.  If $R'$ is a $W$-critical extension of $R\subsetneq V(G)$ with core $X$, then
\begin{equation}\label{eq:ExtensionDrop}
\rho_G(R')\leq\rho_G(R)+\rho(W)-\left(\rho(K_{|X|})+\delta T(K_{|X|})-\delta|X|\right).
\end{equation}
\end{lemma}
\begin{proof}
Suppose that $G$ is a $k$-critical graph with proper vertex subset $R$ and that $G[R]$ is properly $(k-1)$-colored by $\phi$. Let $R'$ be any $W$-critical extension.
The three elements of a graph that contribute to $\varepsilon$-potential are the vertices, the edges, and $T$.
We note that each side of the inequality in Equation \ref{eq:ExtensionDrop} counts the same number of vertices. For the edges, each side of Equation \ref{eq:ExtensionDrop} counts some edges that the other side does not.
Note that only $\rho_G(R')$ includes edges in $G$ from $R$ to $V(W)-X$, only the right side includes edges in $G_{R,\phi}$ from $X$ to $V(W)-X$, and all other edges are accounted for by both sides. However, each edge from $X$ to $V(W)-X$ corresponds to at least one distinct edge from $R$ to $V(W)-X$, so the negative contribution of edges to the $\varepsilon$-potential is always greater on the left side. In fact, if the $W$-critical extension is $i$-incomplete, then the left side counts exactly $i$ edges more than the left.

Therefore, if Equation \ref{eq:ExtensionDrop} is not satisfied, it can only be because of the contribution of the subgraph-measuring parameter $T$.
We observe that $T(G[R'])\geq T(G[R])+T(W-X)$ and that $T(W-X)\geq T(W)-|X|$ because each $x_i\in X$ could be in at most one subgraph counted by $T(W)$. Therefore, the desired inequality holds.
\renewcommand{\qedsymbol}{$\blacksquare$}
\end{proof}

\begin{coro}\label{cor:ExtensionDrop}
Suppose $G$ is a minimal counterexample to Theorem \ref{thm:Main1}.  If $R'$ is a $W$-critical extension of $R\subsetneq V(G)$ and $R$ is not a clique, then $\rho_G(R')\leq \rho_G(R)-2(k-1)-\delta$.
\end{coro}
\begin{proof}
Let $G$ be a minimal counterexample to Theorem \ref{thm:Main1}.
We aim to maximize the right side of Equation \ref{eq:ExtensionDrop}. 
Because $R$ is not a clique we may assume that $W$ is smaller than $G$.
Therefore $\rho(W)$ follows Theorems \ref{thm:Main1} and \ref{thm:Main2}, depending on whether $W$ is a $k$-Ore graph or not.
It follows that the right side is maximized when $W$ is a $k$-Ore graph and $|X|=1$, so we make those two assumptions as well.
If $W=K_k$, then because $T(W)=T(W-x)$ for $x\in X$ we can ignore the contribution of $\delta|X|$ in Equation 
\ref{eq:ExtensionDrop}.
It follows in this case that 
\[\rho_G(R')\leq\rho_G(R)+(k^2-3k+k\varepsilon-2\delta)-(k^2-k-2+\varepsilon)\]
\[=\rho_G(R)-2(k-1)+(k-1)\varepsilon-2\delta.\]
But recall that $\delta=(k-1)\varepsilon$, so the corollary holds when $W=K_k$.

If $W$ is not $K_k$, then it follows from Theorem \ref{thm:Main2} that
\[\rho_G(R')\leq \rho_G(R)-2(k-1)-\varepsilon+\delta+|V(G)|\varepsilon-\left(2+\frac{|V(G)|-1}{k-1}\right)\delta.\]
Again, because $\delta=(k-1)\varepsilon$ the corollary is proven.
\renewcommand{\qedsymbol}{$\blacksquare$}
\end{proof}

%%%%%%%%%%%%%%%%%%%%%%%%%%%%%%%%%%%%%%%
%%%%%%%%%%%%%%%%%%%%%%%%%%%%%%%%%%%%%%%

\section{Edge-Additions}\label{sec:EdgeAdd}

The goal of this section is to establish Lemma \ref{lem:EdgeAddition} which says that a subgraph of a minimal counterexample to Theorem \ref{thm:Main1} cannot be within $\frac{k-4}{2}$ edges of being a smaller $k$-critical graph.  This will be used to establish structural results in Section 6.

\begin{defn}
A proper vertex subset $R\subsetneq V(G)$ is \emph{$i$-collapsible in $G$} if for all proper $(k-1)$-colorings $\phi$ of $G[R]$ using color set $C$
\begin{equation}\label{eq:i-collapseEq}
\min_{c\in C}\left|\left\{uv\in E(G)\mid u\in\phi^{-1}(C-c)\cap R\text{ and }v\in V(G)-R\right\}\right|\leq i.
\end{equation}
\end{defn}
That is, a proper vertex subset $R$ is $i$-collapsible if there is a ``majority'' color class in $\phi(\partial_GR)$ which covers all but at most $i$ edges from $R$ into $V(G)-R$. Note that the boundary vertices $\partial_GR$ of a 0-collapsible set receive the same color in every proper $(k-1)$-coloring of $R$.
\begin{defn}
Let $G$ be a $k$-critical graph. An \emph{$(i+1)$-edge-addition in $G$} is a set $S$ of at most $(i+1)$ edges such that there exists a $k$-critical graph $H$ with $S\subseteq E(H)$, $H-S\subseteq G$, and $V(H)\subsetneq V(G)$.
\end{defn}
Thus, a 1-edge-addition is a single edge that, when added to $G$, forms a $k$-critical subgraph on fewer vertices than $|V(G)|$.  For $i$-edge-additions with $i>1$, the size of $S$ is more flexible; this is important for making the subsequent arguments efficiently.  In the proof of Lemma \ref{lem:EdgeAddition} we do specify the number of edges in $S$, but this will be controlled inductively rather than semantically. 

\begin{lemma}\label{lem:No2Cut}
A minimal counterexample to Theorem \ref{thm:Main1} %main theorem
does not contain a 2-vertex cutset.
\end{lemma}

\begin{proof}
Let $G$ be a minimal counterexample to Theorem \ref{thm:Main1} %main theorem
and suppose that there exists a 2-vertex cutset $\{x,y\}$. 
Because $G$ is $k$-critical, by Dirac \cite{DiracCritical}, deleting $\{x,y\}$ leaves behind two components $H_1$ and $H_2$ such that $\tilde{G}_1=G-H_2$ is $(k-1)$-colorable by $\phi$ where $\phi(x)=\phi(y)$ and $\tilde{G}_2=G-H_1$ is $(k-1)$-colorable by $\psi$ where $\psi(x)\neq\psi(y)$.
Moreover, because $G$ is $k$-critical there does not exist a proper $(k-1)$-coloring of $\tilde{G}_1$ where $x$ and $y$ receive different colors.  
This fact prevents $x$ and $y$ from having a common neighbor $z$ in $\tilde{G}_2$, as a proper $(k-1)$-coloring of $G-xz$ would be a contradiction.
Therefore $x$ and $y$ have no common neighbors in $\tilde{G}_2$, which implies that $G$ is an Ore composition of $\tilde{G}_1+xy$ and $\tilde{G}_2/$\ul{$xy$}, which we rename $G_1$ and $G_2$ respectively.

Because $G$ is not a $k$-Ore graph, at most one of $G_1$ and $G_2$ is a $k$-Ore graph.
From the definition of an Ore composition, it follows that $\rho(G)=\rho(G_1)+\rho(G_2)-k^2-3k-\varepsilon+\delta\left(T(G_1)+T(G_2)-T(G)\right)$.
Because the following argument does not rely on the distinction between edge-side or split-side, we may assume without loss of generality that  $G_1$ is not a $k$-Ore graph.
Using Lemma \ref{lem:TwithkOre} and the fact that $G_1$ is smaller than $G$, we have
\[\rho(G)\leq\rho(G_2)-2(k-1)-\varepsilon+2\delta.\]
Thus $G_2$ has higher $\varepsilon$-potential than $G$.
As $G$ is a minimal counterexample to Theorem \ref{thm:Main1} and $G_2$ is smaller than $G$, it follows that $G_2$ must be a $k$-Ore graph.

If $G_2\neq K_k$, then, %based on second assertion of Thereom
it follows from Theorem \ref{thm:Main2} that $\rho(G)\leq k(k-3)-2(k-1)+(n-1)\left(\varepsilon-\frac{\delta}{k-1}\right)$ where $n=|V(G_2)|$.
If $G_2=K_k$, then Lemma \ref{lem:TwithkOre} %How Ore comp affects T
gives $T(G_1)+T(G_2)-T(G)\leq1$, so it follows that
$\rho(G)\leq k(k-3)+k\varepsilon-2\delta-2(k-1)-\varepsilon+\delta$.
Because $\delta=(k-1)\varepsilon$ both of these inequalities show that $\rho(G)\leq k(k-3)-2(k-1)$, contradicting that $G$ is a minimal counterexample to Theorem \ref{thm:Main1}.
%!  This sets the value for \delta
\renewcommand{\qedsymbol}{$\blacksquare$}
\end{proof}

\begin{prop}\label{prp:IncompleteToCollapsible}
Let $G$ be a $k$-critical graph. If $R\subsetneq V(G)$ is a proper vertex subset where all $W$-critical extensions of $R$ are spanning, have core size 1, and are at most $i$-incomplete, then $R$ is $i$-collapsible in $G$.
\end{prop}

\begin{proof}
Let $G$ be a $k$-critical graph and suppose that we have a proper vertex subset $R$ such that all $W$-critical extensions of $R$ are spanning, have core size 1, and are at most $i$-incomplete.
Then let $\phi$ be an arbitrary proper coloring of $R$ using color set $[k-1]$ and let $R'$ be a $W$-critical extension using $\phi$.
By hypothesis, $R'=V(G)$.
If we permute the colors of $\phi$ so that the vertex in $X$ corresponds to color class 1,
then each edge from $\phi^{-1}(\{2,3,\ldots,k-1\})\cap R$ to $V(G)-R$ contributes to the incompleteness of the $W$-critical extension.
There are at most $i$ such edges so, by definition, $R$ is $i$-collapsible.
\renewcommand{\qedsymbol}{$\blacksquare$}
\end{proof}

\begin{lemma}\label{lem:CollapseToEdgeAddition}
If $G$ is a minimal counterexample to Theorem \ref{thm:Main1} with an $i$-collapsible subset $R\subsetneq V(G)$ for $i\leq(k-3)/2$, then there is an $(i+1)$-edge-addition in $G$.
\end{lemma}
\begin{proof}
Let $G$ be a minimal counterexample to Theorem \ref{thm:Main1} and let $R\subsetneq V(G)$ be an $i$-collapsible subset for $i\leq (k-3)/2$. 
Suppose, for the sake of contradiction, that there is no $(i+1)$-edge-addition in $G$.
For each $u\in\partial_GR$ let $w(u)=|\{uv\in E(G)\mid v\in V(G)-R\}|$. 
Because $G$ is a $k$-critical graph, $G$ is $(k-1)$-edge-connected and thus $\sum_{u\in\partial_GR}w(u)\geq k-1$.
Let $\partial_GR=\{u_1,\ldots, u_s\}$ and, without loss of generality, assume that $w(u_1)\ge w(u_2)\geq \cdots \geq w(u_s)\geq1$.
\renewcommand{\qedsymbol}{}\end{proof}

\begin{proof}[Case 1] Suppose $w(u_2)+\cdots+w(u_s)\geq i+2$.

This case is the same as Case 2 of Lemma 16 in \cite{kostochkayancey2014}, which shows that, for all proper $(k-1)$-colorings $\phi$ of $G[R]$ using color set $C$ and for any color class $\ell\in C$

\[\sum_{u\in\partial_GR-\phi^{-1}(\ell)}w(u)\geq i+1.\]
However, $R$ is $i$-collapsible so this is a contradiction.
\renewcommand{\qedsymbol}{}
\end{proof}

\begin{proof}[Case 2] Suppose $w(u_2)+\cdots+w(u_s)\leq i+1$.

For $i=0$, this implies that $\{u_1,u_2\}$ is a 2-vertex cutset in $G$ so, by Lemma \ref{lem:No2Cut}, we may assume that $i\geq1$.
Let $S=\{u_1u_j\mid 2\leq j\leq s\}$.
Because we have assumed that there is no $(i+1)$-edge-addition and because $S$ is a set of at most $i+1$ edges, there is a proper coloring $\phi$ of $G[R]+S$ using color set $[k-1]$, and $u_1$ is the unique vertex of $\partial_GR$ in its color class. 
Without loss of generality, let $\phi(u_1)=1$.
Because $i\leq\frac{k-3}{2}$ it follows that $w(u_1)\geq i+1$.
Therefore Equation \ref{eq:i-collapseEq} %i-collapse equation
in the definition of $i$-collapsible can only be witnessed by color $1$.
Because $R$ is $i$-collapsible by hypothesis it follows that $w(u_2)+\cdots+w(u_s)\leq i$.

Let $\psi$ be a proper $(k-1)$-coloring of $G[(V(G)-R)\cup\{u_1\}]$ which uses the same colors as $\phi$ such that $\psi(u_1)=1$ and choose $\psi$ so that the number of edges from $\partial_G R$ to $V(G)-R$ which have endpoints colored the same by $\overline{\psi}:=\psi|_{V(G)-R}\cup\phi|_R$ is minimized.
Since $\overline{\psi}$ is not a proper $(k-1)$-coloring of $G$, we may assume that $\phi(u_p)=2$ and one of its neighbors $x$ in $V(G)-R$ also receives color 2.

We will reach a contradiction by relabeling the colors of $\psi$ to interchange $2$ with another color $\ell$ in such a way that $\overline{\psi}$ now gives $u_px$ differently colored endpoints, and so that no edge from $\partial_GR$ to $V(G)-R$ which previously had differently colored endpoints now has endpoints colored the same.  By showing that such an $\ell$ exists, we contradict our initial choice of $\psi$.

Initially, we consider $k-2$ color candidates for $\ell$, obviously needing to remove color 2 as an option.
We also remove color 1 from consideration, so that $\psi(u_1)=\phi(u_1)$ does not change.
Finally, for each of the at most $i$ edges $u_jv$ from $\partial_G R-\{u_1\}$ to $V(G)-R$ we remove $\phi(u_j)$ if $\phi(u_j)\neq2$ and remove $\psi(v)$ if $\phi(u_j)=2$.
This leaves at least $(k-2)-1-i\geq \frac{k-3}{2}\geq i$ choices.
Recall that $i\geq1$, so there does exist a color $\ell$ which contradicts our initial choice of $\psi$, and completes the proof.
\renewcommand{\qedsymbol}{$\blacksquare$}
\end{proof}

\begin{prop}\label{prp:Core1}
Let $G$ be a minimal counterexample to Theorem \ref{thm:Main1}.  If $R\subsetneq V(G)$ is a proper vertex subset that is not a clique and $\rho_G(R)<\rho(G)+k^2-3k+4-\varepsilon$, then every $W$-critical extension of $R$ has core size 1.
\end{prop}
\begin{proof}
Let $G$ be a minimal counterexample to Theorem \ref{thm:Main1} and let $R\subsetneq V(G)$ be a proper vertex subset that is not a clique such that $\rho_G(R)<\rho(G)+k^2-3k+4-\varepsilon$.  Suppose that $R'$ is a $W$-critical extension with core $X$ where $|X|>1$.
The computation in Corollary \ref{cor:ExtensionDrop} maximized the right side of Equation \ref{eq:ExtensionDrop} by assuming that $|X|=1$.  But if $|X|>1$, then that computation is maximized by assuming $|X|=k-1$ which yields
\[
\rho_G(R')\leq\rho_G(R)+\rho(W)-(2k^2-6k+4+(k-1)\varepsilon-2\delta).
\]
Because $\rho(G)\leq\rho_G(R')$ and using the hypothesis, we get
\[\rho(G)<\rho(G)+k^2-3k+4-\varepsilon+\rho(W)-(2k^2-6k+4+(k-1)\varepsilon-2\delta).
\]
This simplifies to $\rho(W)>k^2-3k+k\varepsilon-2\delta$.  
By Theorem \ref{thm:Main2}, this $\varepsilon$-potential is too high for $W$ to be a $k$-Ore graph.  And because $W$ is smaller than $G$, we reach a contradiction with the minimality of $G$.
\renewcommand{\qedsymbol}{$\blacksquare$}
\end{proof}

\begin{lemma}\label{lem:No1EA}
Let $G$ be a minimal counterexample to Theorem \ref{thm:Main1}. There is no 1-edge-addition in $G$.
\end{lemma}
\begin{proof}
Let $G$ be a minimal counterexample to Theorem \ref{thm:Main1} and suppose that there is a 1-edge-addition in $G$.  Among all 1-edge-additions $S$, pick one that minimizes the order of the $k$-critical graph $H\subseteq G+S$.  Let $R=V(H)$ and let $R'$ be a $W$-critical extension of $R$. 
Now $\rho_G(R)\leq \rho(H)+2(k-1)+\delta$ and, because $R$ is not a clique, Corollary \ref{cor:ExtensionDrop} implies that $\rho_G(R')\leq \rho(H)$
It follows that $H$ must be a $k$-Ore graph, as otherwise $H$ is smaller than $G$ and $\rho_G(R')<\rho(G)$ which is not possible.

The $k$-Ore graph with largest $\varepsilon$-potential is $K_k$ so we have
\[\rho_G(R)\leq[k(k-3)+k\varepsilon-2\delta]+2(k-1)+\delta<\rho(G)+2(k-1)+k\varepsilon-\delta+2(k-1).\]
By Proposition \ref{prp:Core1}, and because $4(k-1)+k\varepsilon-\delta<k^2-3k+4-\varepsilon$ for all $k\geq6$, the core of $R'$ has size 1.  Corollary \ref{cor:ExtensionDrop} implies that $\rho_G(R')<\rho(G)+2(k-1)+k\varepsilon-2\delta$.
Note that $R'$ must be complete because otherwise the right side of this inequality would be at least $2(k-1)$ lower, and we would again have $\rho_G(R')<\rho(G)$.  Further, $R'$ must be spanning because otherwise there exists a vertex subset $R''$ such that $\rho_G(R'')<\rho(G)$.  Therefore, $R$ is $0$-collapsible in $G$ by Proposition \ref{prp:IncompleteToCollapsible}.

By definition,  in every proper $(k-1)$-coloring of $G[R]$, each vertex in $\partial_GR$ receives the same color.
If $H$ is $K_k$, then $R=\{u_1,u_2,\ldots,u_k\}$ and we can assume that $\{u_1u_k\}= S$.  We properly $(k-1)$-color $G[R]$ with $\phi$ so that $\phi(u_j)=j$ for $1\le j\le k-1$ and $\phi(u_k)=1$.  Because each vertex in $\partial_GR$ receives the same color, this means that $\{u_1,u_k\}$ is a 2-vertex cutset in $G$ which contradicts Lemma \ref{lem:No2Cut}.  

Therefore $H$ is an Ore composition of two $k$-Ore graphs $H_1$ and $H_2$ with overlap vertices $\{a,b\}$. Note that $S$ must be on the edge-side of the composition---that is $S\subseteq E(H_1)$---because otherwise $\{ab\}$ is a 1-edge-addition that contradicts our choice of $S$.  By Lemma \ref{lem:No2Cut} the set $\{a,b\}$ cannot be a cutset in $G$ so there must be $u,v\in\partial_GR-\{a,b\}$ such that $u\in V(H_1)$ and $v\in V(H_2)\cap G$.
If any proper $(k-1)$-coloring $\phi$ of $G[R]$ has $\phi(u)\notin\{\phi(a),\phi(b)\}$, then we can relabel the colors on $H_1$ so that $\phi(u)\neq\phi(v)$.
This contradicts the fact that $R$ is 0-collapsible.
So without loss of generality, we may assume that $\phi(u)=\phi(a)$.
Let $P=(V(H_2)\cap G)\cup\{a,b\}$.
Now either $\psi(v)=\psi(a)$ in every proper $(k-1)$-coloring $\psi$ of $G[P]$ or we can produce a proper $(k-1)$-coloring of $R$ where $u$ and $v$ receive different colors.
Thus $av$ is a 1-edge-addition that yields a $k$-critical subgraph of order at most $|V(H_2)|+1$ which contradicts our choice of $S$.
\renewcommand{\qedsymbol}{$\blacksquare$}
\end{proof}

\begin{coro}\label{cor:NoDiamond}
Let $G$ be a minimal counterexample to Theorem \ref{thm:Main1}.  For any subgraph $H\subseteq G$, there is no diamond of $H$.  Further, if there is an emerald $D$ of $H$, then there exists a vertex $z\in V(G)-V(D)$ such that $xz\in E(G)$ for each $x\in V(D)$ with $\deg_G(x)=k-1$.  Therefore, there is no emerald of $G$.
\end{coro}
\begin{proof}
Let $G$ be a minimal counterexample to Theorem \ref{thm:Main1}, and let $H$ be a subgraph of $G$.  If $D$ is a diamond of $H$ with endpoints $\{u,v\}$, then $\{uv\}$ is a 1-edge-addition in $G$ which contradicts Lemma \ref{lem:No1EA}.  So we may assume that $D$ is an emerald of $H$.

Note that $\deg_D(x)=k-2$ for each $x\in V(D)$ so each such $x$ is adjacent in $G$ to at least one vertex $V(G)-V(D)$.  If there is at most one $x\in V(D)$ with $\deg_G(x)=k-1$, then the corollary is trivially true.
Suppose then, for the sake of contradiction, that $x,y$ are vertices in $D$ with $\deg_G(x)=\deg_G(y)=k-1$ and $a,b$ are vertices in $V(G)-V(D)$ such that $\{ax,by\}\subseteq E(G)$ and $a\neq b$.  For any proper $(k-1)$-coloring $\phi$ of $G-\{x\}$ it must be the case that the neighbors of $x$ all receive distinct colors.  If we could recolor $y$ using $\phi(a)$, then $\phi$ would extend to all of $G$ which is a contradiction.  Therefore, $\phi(b)$ must be the same color as $\phi(a)$.  But now $\{ab\}$ is a 1-edge-addition in $G$ which contradicts Lemma \ref{lem:No1EA}.

Lastly, if $D$ is an emerald of $G$, then the vertex $z$ guaranteed by the above argument makes a $K_k$ subgraph in $G$ which is not possible in a minimal counterexample to Theorem \ref{thm:Main1}.
\renewcommand{\qedsymbol}{$\blacksquare$}
\end{proof}

\begin{lemma}\label{lem:EdgeAddition}
In a minimal counterexample $G$ to Theorem \ref{thm:Main1}, %main theorem
there is no proper vertex subset $R$ where $R$ is not a clique and $\rho_G(R)<\rho(G)+2(i+1)(k-1)+\delta$ for $1\leq i\leq\frac{k-4}{2}.$ Further, $G$ does not have an $i$-edge-addition for $1\leq i\leq\frac{k-4}{2}$.
\end{lemma}

\begin{proof}
Let $G$ be a minimal counterexample to Theorem \ref{thm:Main1}. We will show first that a subset of the given $\varepsilon$-potential implies that there is an $i$-edge-addition in $G$, and then prove inductively that there are no $\frac{k-4}{2}$-edge-additions in $G$.  First note that, by Corollary \ref{cor:ExtensionDrop}, there is no proper subset that is not a clique and has $\varepsilon$-potential less than $\rho(G)+2(k-1)+\delta$.

\begin{claim}\label{clm:BottomOfBucket}
For each $i$ with $1\leq i\leq\frac{k-4}{2}$ if $G$ has no proper vertex subset that is not a clique with $\varepsilon$-potential less than $\rho(G)+2i(k-1)+\delta$, $R$ is a proper vertex subset that is not a clique, and $\rho_G(R)<\rho(G)+2(i+1)(k-1)+\delta$, then every $W$-critical extension of $R$ is spanning, has core size 1, and is at most $(i-1)$-incomplete. Further, there is an $i$-edge-addition in $G$.
\end{claim}

\begin{proof}[Proof of Claim]

Given $i$, where $1\leq i\leq\frac{k-4}{2}$, suppose that $G$ has no proper vertex subset that is not a clique with $\varepsilon$-potential less than $\rho(G)+2i(k-1)+\delta$ and let $R$ be a proper vertex subset that is not a clique and $\rho_G(R)<\rho(G)+2(i+1)(k-1)+\delta$.
For $i\leq\frac{k-4}{2}$, this implies that $\rho_G(R)<\rho(G)+k^2-3k+2+\delta$ so every $W$-critical extension $R'$ has core size 1 by Proposition \ref{prp:Core1}.

By Corollary \ref{cor:ExtensionDrop} %extensions go down
we also have $\varepsilon$-potential $\rho_G(R')\leq\rho_G(R)-2(k-1)-\delta<\rho(G)+2i(k-1)$. 
By the hypothesis of the claim, $R'$ must be all of $V(G)$.
Also, $R'$ can be at most $(i-1)$-incomplete as otherwise the right side of the inequality would be at least $2i(k-1)$ lower and we would have $\rho_G(R')<\rho(G)$, which is not possible.
By Proposition \ref{prp:IncompleteToCollapsible} and Lemma \ref{lem:CollapseToEdgeAddition}, $R$ is $(i-1)$-collapsible in $G$ and hence there is an $i$-edge-addition in $G$.
\end{proof}

Suppose now that there is an $i$-edge-addition in $G$.  We will prove inductively that any $i$ with $1\leq i\leq\frac{k-4}{2}$ gives a contradiction.  Lemma \ref{lem:No1EA} shows that $i\neq1$.  We may assume that there is no $(i-1)$-edge-addition, so by Claim \ref{clm:BottomOfBucket} there is no proper vertex subset $R$ with $\rho_G(R)<\rho(G)+2i(k-1)+\delta$.
Note that this inductive hypothesis guarantees that $|S|=i$.  Because each edge of $S$ might contribute to $T(H)$, we have
$\rho_G(R)\leq\rho(H)+2i(k-1)+i\delta$.
Among all $i$-edge-additions $S$, we will choose one that minimizes the order of the $k$-critical graph $H\subseteq G+S$. 
\renewcommand{\qedsymbol}{}
\end{proof}
\begin{proof}[Case 1] $H$ is not a $k$-Ore graph.

Because $H$ is smaller than the minimal counterexample $G$, we have $\rho(H)<\rho(G)$.  Thus, we bound the $\varepsilon$-potential of $R$ by $\rho_G(R)<\rho(G)+2i(k-1)+i\delta<\rho(G)+2(i+1)(k-1)+\delta.$  By Claim \ref{clm:BottomOfBucket}, every $W$-critical extension of $R$ is spanning, has core size 1, and is at most $(i-1)$-incomplete. 
Further, there must be some $W$-critical extension $R'$ that is exactly $(i-1)$-incomplete.
Otherwise, Proposition \ref{prp:IncompleteToCollapsible} implies that $R$ is $(i-2)$-collapsible and there is an $(i-1)$-edge-addition in $G$ by Lemma \ref{lem:CollapseToEdgeAddition}.

Choose such a $(i-1)$-incomplete $W$-critical extension $R'$.
Using Lemma \ref{lem:PotentialUnderExtension} and the $(i-1)$-incompleteness of $R'$ we bound the $\varepsilon$-potential as follows:
\[\rho_G(R')<[\rho(G)+2i(k-1)+i\delta] + \rho(W)-2(i-1)(k-1)-(k^2-k-2+\varepsilon-\delta)\]
But $R'=V(G)$, so this implies that $\rho(W)>k^2-3k-(i+1)\delta+\varepsilon$.
Because $W$ is smaller than $G$, this contradicts the minimality of $G$ unless $W$ is a $k$-Ore graph.
By Lemma \ref{lem:DiamondAwayFromPoint}, there is a subgraph $D\subseteq W-X\subseteq G$ which is an emerald of $W$.
Corollary \ref{cor:NoDiamond} gives a vertex $z\in V(G)-V(D)$ such that $xz\in E(G)$ for each $x\in V(D)$ with $\deg_G(x)=k-1$.
Because $R'$ is spanning, the only edges in $G$ that can cause $\deg_G(x)>k-1$ for $x\in V(D)$ are edges from $x$ to $R$ which do not correspond to an edge used in $W$.  These edges contribute to the incompleteness of a $W$-critical extension, so $z$ has at most $(i-1)$ non-neighbors in $D$.  Adding these edges yields a $K_k$, which contradicts the inductive hypothesis.
\renewcommand{\qedsymbol}{}
\end{proof}
\begin{proof}[Case 2] $H$ is a $k$-Ore graph but is not $K_k$

If $H$ is a $k$-Ore graph that is not $K_k$, then $\rho(H)\leq k(k-3)+\varepsilon-2\delta$, and because $k(k-3)<\rho(G)+2(k-1)$ we also have $\rho(R)<\rho(G)+2(i+1)(k-1)+(i-2)\delta+\varepsilon$.  For $i\leq\frac{k-4}{2}$, this upper bound satisfies the hypothesis of Proposition \ref{prp:Core1} so every $W$-critical extension $R'$ has core size 1.
Corollary \ref{cor:ExtensionDrop} %extensions go down 
implies that $\rho_G(R')<\rho(G)+2i(k-1)+(i-3)\delta+\varepsilon$.
For $i=2$, the $W$-critical extension $R'$ is at most 1-incomplete because otherwise the right side is lowered by at least $4(k-1)$ and we get $\rho_G(R')<\rho(G)-\delta+\varepsilon$.  This implies that $\rho_G(R')<\rho(G)$ which is not possible.
Note that for $i>3$ it is possible that $R'$ is $i$-incomplete according to this bound, but cannot be $j$-incomplete for $j\geq i+1$.

First, suppose that $H$ is an Ore composition of two $k$-Ore graphs $H_1$ and $H_2$ with overlap vertices $\{a,b\}$.
Note that all edges of $S$ must be on the edge-side of the composition $H_1$ as otherwise adding $ab$ to $S\cap E(H_1)$ is an $i$-edge-addition that contradicts our choice of $S$.  Thus $H_2-\underline{ab}\subseteq G$.
By Lemma \ref{lem:DiamondAwayFromPoint}, there is a subgraph $D\subseteq H_2-\underline{ab}\subseteq G$ which is an emerald of $H_2$.  Corollary \ref{cor:NoDiamond} gives a vertex $z\in V(G)-V(D)$ such that $xz\in E(G)$ for each $x\in V(D)$ with $\deg_G(x)=k-1$.  
For each $x\in V(D)$, we have $\deg_{H_2}(x)=\deg_{H}(x)=k-1$, so $z\in V(H)$ and either $x\in N_G(z)$ or $x\in\partial_GR$.
But adding the edges $\{yz\mid y\in V(D)\cap\partial_GR\}$ creates a $K_k$ subgraph so by the inductive hypothesis and our choice of $S$ it follows that $|V(D)\cap\partial_GR|\geq i+1$.  

By Lemma \ref{lem:No2Cut}, $\{a,b\}$ is not a cutset so there is some $u\in \partial_GR-\{a,b\}$ such that $u\in V(H_1)$.  Let $\phi$ be a proper $(k-1)$-coloring of $G[R]$, with the colors permuted so that the vertex in the core $X$ of the $W$-critical extension corresponds to color class 1.
Thus each edge from $\phi^{-1}(\{2,3,\ldots,k-1\})\cap R$ to $V(G)-R$ contributes to the incompleteness of the $W$-critical extension.  In the case where $i=2$, $R'$ is at most $1$-incomplete so $|V(D)\cap\partial_GR|\leq2$.  This contradicts our earlier bound on this set.  Therefore we may assume $i>3$ for the rest of this case.  Because $R'$ is at most $i$-incomplete $|V(D)\cap\partial_GR|\leq i+1$.  This implies that $|V(D)\cap\partial_GR|=i+1$, $R'$ is exactly $i$-incomplete, and that $\phi(u)=1$.

If $\phi(u)\notin\{\phi(a),\phi(b)\}$, then we can relabel the colors on $H_1$ only so that $\phi(u)$ is not given to any vertex in $V(D)\cap\partial_GR$.  Because all $W$-critical extensions of $R$ have a core of size 1, this new coloring would give a $W$-critical extension that is $i+1$ incomplete which is a contradiction.  Therefore it must be the case that, for every proper $(k-1)$-coloring of $G[R]$, $\phi(u)\in\{\phi(a),\phi(b)\}$.  This means that $\{ua,ub\}$ is a 2-edge-addition which contradicts the fact that $i>3$.
\renewcommand{\qedsymbol}{}
\end{proof}

\begin{proof}[Case 3] $H$ is $K_k$.

For this case, we further refine our bound $\rho_G(R)\leq\rho(H)+2i(k-1)+t\delta$.  We do not know how many edges of $S$ contribute to $T(H)$, but $t\leq 2$.
The $\varepsilon$-potential of $H$ is $\rho(H)= k(k-3)+k\varepsilon-2\delta$ and so $\rho(R)<\rho(G)+2(i+1)(k-1)+(t-2)\delta+k\varepsilon$.  For $i\leq\frac{k-4}{2}$, this upper bound satisfies the hypothesis of Proposition \ref{prp:Core1} so every $W$-critical extension $R'$ has core size 1.
Corollary \ref{cor:ExtensionDrop} %extensions go down 
implies that $\rho_G(R')<\rho(G)+2i(k-1)+(t-3)\delta+k\varepsilon$.
Note that $R'$ is at most $i$-incomplete, as otherwise $\rho_G(R')<\rho(G)$.

We label $R=V(H)=\{u_1,\ldots,u_k\}$ so that $u_1u_k\in S$ and
properly $(k-1)$-color $G[R]$ with $\phi$ so that $\phi(u_j)=j$ for $1\leq j\leq k-1$ and $\phi(u_k)=1$.
Because $\deg_G(u_j)\geq k-1$ each vertex $u_j\in R$ has at least as many edges in $G$ from $u_j$ to $V(G)-R$ as the number of edges of $S$ incident with $u_j$. 
Any color class that not incident to an edge in $S$ will miss at least $i+1$ endpoints of $S$. So for $R'$ to be at most $i$-incomplete, the vertex in the core $X$ corresponds to color class 1 and every edge in $S$ must incident to at least one of $u_1$ or $u_k$.
If $u_1u_2$ and $u_ku_3$ are both in $S$, then switching the colors on $u_2$ and $u_k$ give a proper $(k-1)$-coloring of $G[R]$ where every color class is not incident to at least one edge in $S$, which is a contradiction.
Thus we may assume that, without loss of generality, either $|S|=3$ and $S$ forms a triangle subgraph or $S$ forms a star subgraph with $u_1$ as the center.  In either case, $t\leq1$ because $G[R]$ has a $K_{k-2}$ subgraph.
Thus the bound given by Corollary \ref{cor:ExtensionDrop} is $\rho_G(R')<\rho(G)+2i(k-1)-2\delta+k\varepsilon$.  With this bound, $R'$ cannot be $i$-incomplete because $2\delta>k\varepsilon$.

If $S$ is a triangle, let $S=\{u_1u_k,u_1u_2,u_ku_2\}$.
Because $R'$ is at most $2$-incomplete, by changing which two vertices of $\{u_1,u_2,u_k\}$ have the same color in a proper $(k-1)$-coloring of $G[R]$, it follows that $\partial_GR=\{u_1,u_2,u_k\}$ and each of these vertices has exactly two edges to $V(G)-R$.
Thus there are 6 edges from $R$ to $V(G)-R$.
However, $i\leq\frac{k-4}{2}$ and $i=3$ imply that $k\geq10$, which is a contradiction as $k$-critical graphs are $(k-1)$-edge-connected.

Suppose instead that $S$ is a star with $u_1$ as the center.
Because $R'$ is at most $(i-1)$-incomplete, every leaf $u_j$ of the star has exactly one neighbor in $V(G)-R$, say $y_j$.
Consider the graph $F=G-\{u_2,\ldots,u_k\}$.
No proper $(k-1)$-coloring $\psi$ of $F$ can be extended to all of $G$, so it follows that $\psi(u_1)=\psi(y_j)$ for each $j$ where $u_j$ is a leaf of $S$.
Thus $u_1y_j$ is a 1-edge-addition in $G$, which contradicts Lemma \ref{lem:No1EA}.
\renewcommand{\qedsymbol}{$\blacksquare$}
\end{proof}

%%%%%%%%%%%%%%%%%%%%%%%%%%%%%%%%%%%%%%%
%%%%%%%%%%%%%%%%%%%%%%%%%%%%%%%%%%%%%%%%%

\section{Cloning}\label{sec:Cloning}

Cloning is a reduction operation that will help us understand the structures that exist near vertices of degree $k-1$ in a minimal counterexample to Theorem \ref{thm:Main1}.

\begin{defn}
Let $G$ be a $k$-critical graph with $xy\in E(G)$ such that $\deg_G(x)=k-1$. We define \emph{cloning $x$ with $y$} to mean constructing a new graph $G_{y\rightarrow x}$ such that $V(G_{y\rightarrow x})=V(G)\cup\{\tilde{x}\}-\{y\}$ and $E(G_{y\rightarrow x})=E(G-y)\cup\{\tilde{x}v\mid v\in N_G(x)\}\cup\{\tilde{x}x\}$.
\end{defn}
Thus the vertex $y$ is replaced with the new vertex $\tilde{x}$, which is a copy of $x$.
Below we define the notion of a cluster, which was introduced in \cite{kostochkayancey2014}. 

\begin{defn}
A \emph{cluster} is a maximal set $R\subseteq V(G)$ such that $\deg_G(x)=k-1$ for every $x\in R$ and $N_G[x]=N_G[y]$ for every pair $x,y\in R$.
\end{defn}
Note that if $x\in V(G)$ is in a cluster $C_x$ and $xy\in E(G)$, then in $G_{y\rightarrow x}$ the new vertex $\tilde{x}$ is added to the cluster $C_x$.  Further, if $x'$ is a second vertex in $C_x$, then $G_{y\rightarrow x'}=G_{y\rightarrow x}$.
If $y$ is already in $C_x$, then $G_{y\rightarrow x}=G$. If $y$ is not in $C_x$, then $G_{y\rightarrow x}$ is smaller than $G$ except in the case where $\deg(y)=k-1$ and $G_{y\rightarrow x}$ is $k$-critical. In this case, we further need $y$ to be in a cluster of size at most $|C_x|$ for $G_{y\rightarrow x}$ to be smaller than $G$.

\begin{lemma}\label{lem:NoColorClone}
If $G$ is a $k$-critical graph where $xy\in E(G)$, $x$ is in a cluster of size $s$, and $\deg_G(y)\leq k-2+s$, then $G_{y\rightarrow x}$ is not $(k-1)$-colorable.
\end{lemma}
\begin{proof}
Let $G$ be a $k$-critical graph and let $xy\in E(G)$ such that $x$ is in a cluster $C_x$ of size $s$ and $\deg_G(y)\leq k-2+s$.
Suppose, for the sake of contradiction, that $\phi$ is a proper $(k-1)$-coloring of $G_{y\rightarrow x}$. 
Let $\psi$ be the partial proper coloring of $G$ obtained by copying $\phi(u)$ for every $u\in V(G)-\{y\}$.
Because $y$ has at most $k-2$ neighbors outside of $C_x$ we can choose $\psi(y)$ to be a color distinct from these neighbors.  But now $\psi(y)=\psi(z)$ for some vertex $z\in C_x$ because $G$ is $k$-critical.  Without loss of generality, we can assume that $z=x$.
We recolor $x$ so that $\psi(x):=\phi(\tilde{x})$ and now $\psi$ is a proper $(k-1)$-coloring of $G$, which is a contradiction.
\renewcommand{\qedsymbol}{$\blacksquare$}
\end{proof}

\begin{lemma}\label{lem:Clone}
Suppose that $G$ is a minimal counterexample to Theorem \ref{thm:Main1} and $xy\in E(G)$ such that (1) $x$ is in a cluster $C_x$ of size $s$, (2) $\deg_G(y)\leq k-2+s$, and (3) if $y$ is in a cluster $C_y$, then $C_y\neq C_x$ and $|C_y|=t\leq s$. Then for any $k$-critical subgraph $H\subseteq G_{y\rightarrow x}$ either $H$ is a $k$-Ore graph or $H=G_{y\rightarrow x}$. Moreover, $H=G_{y\rightarrow x}$ is only possible if $\deg_G(y)=k-1$.
\end{lemma}
\begin{proof}
Let $G$ be a minimal counterexample to Theorem \ref{thm:Main1} and let $xy\in E(G)$ such that (1) $x$ is in a cluster $C_x$ of size $s$, (2) $\deg_G(y)\leq k-2+s$, and (3) if $y$ is in a cluster $C_y$, then $C_y\neq C_x$ and $|C_y|=t\leq s$.
Let $G_{y\rightarrow x}$ be the graph obtained by cloning $x$ with $y$. By Lemma \ref{lem:NoColorClone} $G_{y\rightarrow x}$ is not $(k-1)$-colorable, so there exists a $k$-critical subgraph $H\subseteq G_{y\rightarrow x}$. Note that condition (3) ensures that $H$ is smaller than $G$.
Suppose that $H$ is not a $k$-Ore graph; we will see that this either leads to contradiction, or implies that $\deg_G(y)=k-1$ and $H=G_{y\rightarrow x}$. 

We let $R=V(H)-\{\tilde{x}\}$ and note that $R$ is not a clique because $H$ is not a $k$-Ore graph. One can compute that $\rho_G(R)\leq\rho(H)+k^2-3k+4-\varepsilon+\delta$.
Let $R'$ be a $W$-critical extension of $R$ with core $X$. Because $\rho(G)\leq\rho_G(R')$ and because $H$ is smaller than $G$ but is not a $k$-Ore graph, Lemma \ref{lem:PotentialUnderExtension} %PotentialUnderExtension
yields the inequality
\begin{equation}\label{eq:Clone1}
[\rho(K_{|X|})+\delta T(K_{|X|})-\delta|X|]-k^2+3k-4+\varepsilon-\delta<\rho(W).
\end{equation}
For $1<|X|<k-1$, this gives $W$ an $\varepsilon$-potential that is too high for $W$ to be a $k$-Ore graph by Theorem \ref{thm:Main2}.  Because $W$ is smaller than $G$, this contradicts the minimality of $G$.

Suppose now that $|X|=k-1$. Then Observation \ref{fct:CompletePotential} implies that $k^2-3k+k\varepsilon-k\delta<\rho(W)$, which is a contradiction unless $W$ is a $k$-Ore graph. When $W$ is a $k$-Ore graph, Equation \ref{eq:Clone1} is almost tight; more specifically, the difference between the two sides is less than $2(k-1)$. Therefore, it follows that $R'$ is a spanning and complete $W$-critical extension, because otherwise the right side is lowered by at least $2(k-1)$.

If $W$ is not $K_k$, Lemma \ref{lem:DiamondAwayFromGroup} implies that $D\subseteq W$ is a diamond or emerald of $W$ disjoint from $X$. Because $W-X\subseteq G$, Corollary \ref{cor:NoDiamond} implies that $D$ is an emerald of $W$.  But $R'$ is a spanning and complete extension, so $\deg_G(x)=k-1$ for each $x\in V(D)$. Thus $D$ is an emerald of $G$, which contradicts Corollary \ref{cor:NoDiamond}.  Therefore we may assume that $W$ is $K_k$, and it follows that $V(G)=R\cup\{y\}$. Thus $T(H)$ and $T(G)$ can differ by at most 1, and it must be that $|E(H)|=|E(G)|$. This implies that $\deg_G(y)=k-1$ and $H=G_{y\rightarrow x}$.

Suppose instead that $|X|=1$. We claim that $R'$ must be a spanning $W$-critical extension that is at most $\frac{k-4}{2}$-incomplete. For an $i$-incomplete $W$-critical extension, we have
\begin{equation}\label{eq:Clone2}
\rho(G)\leq\rho_G(R')\leq\rho(H)+\rho(W)-2k+6-2i(k-1)-2\varepsilon+2\delta,
\end{equation}
which because $H$ is smaller than $G$ yields
\begin{equation}\label{eq:Clone3}
2k-6+2i(k-1)+2\varepsilon-2\delta<\rho(W).
\end{equation}
Lemma \ref{lem:EdgeAddition} implies that any proper vertex subset that is not a clique must have $\varepsilon$-potential at least $\rho(G)+k^2-3k+2+\delta$.  If $R'$ is not spanning, the left side of Equation \ref{eq:Clone2}, and subsequently Equation \ref{eq:Clone3}, can be increased by $k^2-3k+2+\delta$. Thus $k^2-k-4+2\varepsilon-\delta>\rho(W)$, which  contradicts either Theorem \ref{thm:Main2} or the minimality of $G$. So we may assume that $R'$ is spanning. 
If $i\geq \frac{k-3}{2}$, then we get $k^2-2k-3+2\varepsilon-2\delta<\rho(W)$ which also contradicts either
Theorem \ref{thm:Main2} or the minimality of $G$. Therefore $R'$ is spanning and is at most $\frac{k-4}{2}$-incomplete.
In fact, there must be a particular $W$-critical extension $R'$ that is $\frac{k-4}{2}$-incomplete or $\frac{k-5}{2}$-incomplete, as otherwise $R$ is $\frac{k-6}{2}$-collapsible and then there exists a $\frac{k-4}{2}$-edge-addition in $G$ by Proposition \ref{prp:IncompleteToCollapsible} and Lemma \ref{lem:CollapseToEdgeAddition}, which contradicts Lemma \ref{lem:EdgeAddition}.

We choose such an $i$-incomplete $W$-critical extension $R'$ for $i\in\left\{\frac{k-4}{2},\frac{k-5}{2}\right\}$.  Now Equation \ref{eq:Clone3} becomes $k^2-4k-1+2\varepsilon-2\delta<\rho(W)$.  This $\varepsilon$-potential does not match the conclusion of Theorem \ref{thm:Main1} so $W$ must be a $k$-Ore graph by the minimality of $G$.
As $W$ is a $k$-Ore graph, Lemma \ref{lem:DiamondAwayFromPoint} implies that $D\subseteq W$ is a diamond or emerald of $W$ disjoint from $X$. 
Corollary implies that \ref{cor:NoDiamond} $D$ is an emerald of $W$ and there must exist a vertex $z$ in $V(G)-V(D)$ such that $xz\in E(G)$ for each $x\in V(D)$ with $\deg_G(x)=k-1$.
However, $R'$ is at most $\frac{k-4}{2}$-incomplete, so there are at most $\frac{k-4}{2}$ vertices of $D$ that are not adjacent to $z$.  The set of edges from these vertices to $z$ is a $\frac{k-4}{2}$-edge-addition, which contradicts Lemma \ref{lem:EdgeAddition}.
\renewcommand{\qedsymbol}{$\blacksquare$}
\end{proof}

To talk about the different outcomes of a cloning operation, we introduce the following terminology.
\begin{defn}
A \emph{gadget}, $H^\circ$, is a graph obtained from a $k$-Ore graph $H$ by deleting a vertex $x$ of degree $k-1$ in a cluster of size at least $2$. Note that the requirement of cluster size prevents $x$ from being an overlap vertex of an Ore composition. A \emph{gadget of $G$} is a subgraph of $G$ that is a gadget.
\end{defn}

\begin{defn}
A \emph{key vertex} of a $k$-Ore graph $H$ is a vertex $x$ such that, whenever $H$ is an Ore composition of two graphs $H_1$ and $H_2$ with overlap vertices $\{a,b\}$, $x\in V(H_1)-\{a,b\}$.  That is, $x$ is on the edge-side of the composition and is not an overlap vertex.  A \emph{key vertex} of a gadget is a vertex which is a key vertex of the corresponding $k$-Ore graph.
\end{defn}

\begin{coro}\label{cor:CloneName}
Suppose that $G$ is a minimal counterexample to Theorem \ref{thm:Main1} and $xy\in E(G)$ such that (1) $x$ is in a cluster $C_x$ of size $s$, (2) $\deg_G(y)\leq k-2+s$, and (3) if $y$ is in a cluster $C_y$, then $C_y\neq C_x$ and $|C_y|=t\leq s$.
Then $x$ is a key vertex of a gadget of $G$, or $x$ is in a $K_{k-3}$ subgraph of $G$. Moreover, the latter is only possible if $\deg_G(y)=k-1$ and $y$ is not in the $K_{k-3}$ subgraph.
\end{coro}
\begin{proof}
By Lemma \ref{lem:Clone} there is a $k$-critical graph $H\subseteq G_{y\rightarrow x}$. If $H$ is a $k$-Ore graph, then $H-\tilde{x}$ is a gadget of $G$.
Suppose that $H$ is an Ore composition of two $k$-Ore graphs $H_1$ and $H_2$ with overlap vertices $\{a,b\}$.
If $x\in V(H_2)$ or if $x\in\{a,b\}$, then $ab$ is a 1-edge-addition in $G$, which contradicts Lemma \ref{lem:No1EA}.
Because every vertex of $K_k$ is trivially a key vertex, it follows that $x$ is a key vertex of $H-\tilde{x}$.

If $H$ is not a $k$-Ore graph, then $\deg_G(y)=k-1$ by Lemma \ref{lem:Clone} and thus $H$ and $G$ have the same number of edges. However, $H$ is smaller than $G$ because $s\geq t$. Thus $\rho(H)<\rho(G)$ which is only possible if adding $\tilde{x}$ creates either a $K_{k-2}$ or $K_{k-1}$ subgraph of $H$ that doesn't exist in $G$. In either case, $x$ is in a $K_{k-3}$ subgraph of $G$ that does not contain $y$.
\renewcommand{\qedsymbol}{$\blacksquare$}
\end{proof}

To aid with discharging, it is useful to classify the vertices of degree $k-1$ in a minimal counterexample to Theorem \ref{thm:Main1} into three distinct groups.
\begin{defn}
Let $G$ be a minimal counterexample to Theorem \ref{thm:Main1} and suppose that $x\in V(G)$ is a vertex of degree $k-1$. Let $C_x$ be the cluster containing $x$; note that every vertex withing a given cluster is classified into the same group.
\begin{itemize}
\item If $x$ is a key vertex of a gadget or is in a $K_{k-3}$ subgraph, then we call $x$ a \emph{structure-vertex}.
\item If $x$ is not a structure-vertex and is adjacent to a vertex $y$ which belongs to a distinct cluster $C_y$, then we call $x$ a \emph{near-vertex}. Note that Corollary \ref{cor:CloneName} implies that $y$ is necessarily a structure-vertex and that $|C_x|<|C_y|$.
\item If $x$ is not a structure-vertex and every neighbor of $x$ with degree $k-1$ is in $C_x$, then we call $x$ a \emph{lone-vertex}. Note that $|C_x|\leq k-4$, or $x$ would be a structure-vertex.
\end{itemize}
\end{defn}

\begin{lemma}
Suppose that $G$ is a minimal counterexample to Theorem \ref{thm:Main1} and that $x$ is a structure-vertex in $G$. Then $x$ cannot be adjacent to two near-vertices $y$ and $z$ with $C_y\neq C_z$.
\end{lemma}
\begin{proof}
Let $G$ be minimal counterexample to Theorem \ref{thm:Main1} and suppose that $x$ is a structure-vertex with two near-vertex neighbors $y$ and $z$ such that $C_y\neq C_z$.
If $yz\in E(G)$, then Corollary \ref{cor:CloneName} implies that either $y$ or $z$ is a structure-vertex, which is a contradiction.
Therefore we conclude that $yz\notin E(G)$ and consider $G_{x\rightarrow z}$.
By Lemma \ref{lem:NoColorClone} there is a $k$-critical subgraph $H\subseteq G_{x\rightarrow z}$, and $H$ cannot include the vertex $y$.
Therefore $|V(H)|<|V(G_{x\rightarrow z})|=|V(G)|$ and we know that $H$ is smaller than $G$.
This replaces the need for condition (3) of Lemma \ref{lem:Clone} and Corollary \ref{cor:CloneName} and so it follows that $z$ is a structure-vertex, which contradicts the fact that it is a near-vertex.
\renewcommand{\qedsymbol}{$\blacksquare$}
\end{proof}

\begin{lemma}\label{lem:AdjacentBound}
In a minimal counterexample $G$ to Theorem \ref{thm:Main1}, let $x$ and $y$ be adjacent vertices such that $\deg_G(x)=k-1$ and $N_G[x]$ is not a subset of $N_G[y]$. Then $\deg_G(y)\geq|N_G(x)\cap N_G(y)|+1+\frac{k-3}{2}$.
\end{lemma}
\begin{proof}
Let $G$ be a minimal counterexample to Theorem \ref{thm:Main1} and let $x$ and $y$ be adjacent vertices such that $\deg_G(x)=k-1$ and $w\in N_G[x]-N_G[y]$.
In any proper $(k-1)$-coloring $\phi$ of $G-x$, the vertices of $N_G(x)$ all receive distinct colors.
Therefore, some vertex of $N_G[y]-N_G[x]$ must be in the same color class as $w$ and
adding the edge set $S=\{wu_i\mid u_i\in N_G[y]-N_G[x]\}$ to $G-x$ creates a $k$-critical subgraph.
Using Lemma \ref{lem:EdgeAddition} we get $|S|\geq\frac{k-3}{2}$, and this gives the desired bound on $\deg_G(y)$.
\renewcommand{\qedsymbol}{$\blacksquare$}
\end{proof}

\begin{lemma}\label{lem:KeyBound}
Let $G$ be a minimal counterexample to Theorem \ref{thm:Main1} and suppose that $x$ is a key vertex in a gadget of $G$ such that $\deg_G(x)=k-1$. Then $x$ has at least $\frac{k-3}{2}$ neighbors of degree at least $\frac{3(k-3)}{2}$.
\end{lemma}

\begin{proof}  
Let $G$ be a minimal counterexample to Theorem \ref{thm:Main1} and let $x$ be a vertex of degree $k-1$ which is a key vertex of a gadget $H^\circ$ of $G$.
Let $H$ be the $k$-Ore graph where $H^\circ=H-w$.  If $H$ is an Ore composition of two graphs $H_1$ and $H_2$ with overlap vertices $\{a,b\}$, then $w\notin\{a,b\}$ and $\deg_H(w)=k-1$ by the definition of gadget.  Further, we must have $w\in V(H_1)$ because otherwise $\{ab\}$ is a 1-edge-addition in $G$ which contradicts Lemma \ref{lem:No1EA}.  Therefore if $H'_2$ is the split-side of the composition after separating the split vertex into $a$ and $b$, then $H'_2\subseteq G$.

Proposition \ref{prp:kOreSequence} gives a sequence of $k$-Ore graphs such that $H$ can be viewed as a $K_k$ graph with some edges replaced by suitable split $k$-Ore graphs.  The same sequence of Ore compositions lets us view $H^\circ$ as a $K_{k-1}$ graph $H'$ with some edges replaced by the same split $k$-Ore graphs.  Because each step in the sequence is the edge-side of the subsequent Ore composition, $V(H')\subseteq V(G)$.
The key vertex $x$ is not an overlap vertex for any Ore composition, so $xu\in E(G)$ for each $u\in V(H')-\{x\}$.  Therefore $x$ has one neighbor $z\in V(G)-V(H^\circ)$.
We partition the vertices of $H'$ into two sets $A:=\{u\in V(H')\mid uz\in E(G)\}$ and $B:=V(H')-A$. 
Note that in any proper $(k-1)$-coloring of $H^\circ$, each vertex of $H'$ gets a distinct color.

First, we show that $V(G)=V(H^\circ)\cup\{z\}$ is not possible. Suppose, for sake of contradiction that $V(G)=V(H^\circ)\cup\{z\}$.
If $H=K_k$, then this implies that $G$ is also $K_k$, which is a contradiction.  Therefore $H$ is an Ore composition of two $k$-Ore graphs $H_1$ and $H_2$ with overlap vertices $\{a,b\}$.  
Because $\rho(G)>k(k-3)-2(k-1)$ by hypothesis and $\rho(H)\leq k(k-3)-2\delta+\varepsilon$ by Theorem \ref{thm:Main2}, it follows that $|E(G)|=|E(H)|$ and therefore $\deg_G(z)=k-1$.
By Lemma \ref{lem:No2Cut}, $\{a,b\}$ is not a cutset, so there exists some $zv$ with $v\in V(H_2)$.

Let $\phi$ be a proper $(k-1)$-coloring of $H^\circ$.  If $\phi(v)\notin\{\phi(a),\phi(b)\}$, then it is possible to relabel the colors on split-side vertices only so that $\phi(v)=\phi(x)$.  But this updated coloring would then extend to $z$, as two of $z$'s neighbors share a color.  Therefore, where $H$ is an Ore composition of two $k$-Ore graphs $H_1$ and $H_2$ with overlap vertices $\{a,b\}$, any neighbor $v$ of $z$ with $v\in V(H_2)$ is colored the same as either $a$ or $b$ by any proper $(k-1)$-coloring of $H^\circ$.  This implies that $\{za, zb, ab\}$ is a 3-edge-addition which contradicts Lemma \ref{lem:EdgeAddition} because $k\geq 10$.

Therefore $V(H^\circ)\cup\{z\}$ is a proper subset of $V(G)$.  But if follows from this that $\{zb\mid b\in B\}$ is a $|B|$-edge-addition in $G$.  By Lemma \ref{lem:EdgeAddition}, we have $|B|\geq\frac{k-3}{2}$.  By Lemma \ref{lem:AdjacentBound}, 
$\deg_G(b)\geq|N_G(x)\cap N_G(b)|+1+\frac{k-3}{2}$ for each $b\in B$. If $b$ is adjacent to each vertex in $V(H')$, then $|N_G(x)\cap N_G(b)|=k-3$ and we get one more than the desired bound. For any $u\in V(H')$ that is not in $N_G(b)$, $H$ is an Ore composition of two $k$-Ore graphs $H_1$ and $H_2$ with overlap vertices $\{u,b\}$.  Let $H_2'$ be the split side of the composition after separating the split vertex into $u$ and $b$; note that $H_2'\subseteq G$.  In any proper $(k-1)$-coloring of $H_2'$, different colors are given to $u$ and $b$ and thus $\{uv\mid v\in N_{H_2'}(b)\}$ is a $|N_{H_2'}(b)|$-edge-addition in $G$.  By Lemma \ref{lem:EdgeAddition}, it follows that $|N_{H_2'}(b)|\geq\frac{k-3}{2}$.  However, the vertices in $|N_{H_2'}(b)|$ may also include the $\frac{k-3}{2}$ vertices in $N_G(b)-N_G(x)$ counted by Lemma \ref{lem:AdjacentBound}.  Therefore, we conclude that 
$\deg_G(b)\geq\frac{3(k-3)}{2}$.
\renewcommand{\qedsymbol}{$\blacksquare$}
\end{proof}

\begin{lemma}\label{lem:Kk3Bound}
If $x$ is in a $K_{k-3}$ subgraph $D\subseteq G$, where $G$ is a minimal counterexample to Theorem \ref{thm:Main1} and $\deg_G(x)=k-1$, then $x$ has at least
$\frac{k-9}{6}$ neighbors of degree at least $\frac{3(k-3)}{2}-1$. Furthermore, if $x$ has a neighbor $y\in V(G)-V(D)$ which is in a different cluster, then $x$ has at least $\frac{k-7}{2}$ neighbors of degree at least $\frac{3(k-3)}{2}-1$.
\end{lemma}

\begin{proof}
Let $G$ be a minimal counterexample to Theorem \ref{thm:Main1} such that $x$ is a vertex of degree $k-1$ in a $K_{k-3}$ subgraph $D$. Let $z_1,z_2,z_3$ be the three neighbors of $x$ in $V(G)-V(D)$. 
We partition the vertices of $D$ into two sets $A:=\{u\in V(D)\mid uz_i\in E(G) \text{ for each } i\in\{1,2,3\}\}$, and $B:=V(D)-A$.
By Lemma \ref{lem:AdjacentBound}, each $b\in B$ has degree at least $(k-5)+1+\frac{k-3}{2}=\frac{3(k-3)}{2}-1$.
It remains to show that $B$ is a large enough set.

The edges $\{z_1z_2,z_1z_3,z_2z_3\}$ and $bz_i$ for each pair $b\in B$, $i\in\{1,2,3\}$ form a $(3+3|B|)$-edge-addition in $G$, so it follows from Lemma \ref{lem:EdgeAddition} that $|B|\geq\frac{k-9}{6}$.
Now suppose without loss of generality that $z_1=y$ is a vertex of degree $k-1$ that is in a different cluster than $x$.
Because there is at least one vertex in $B$, Lemma \ref{lem:AdjacentBound} implies that $\deg_G(z_1)\geq(|A|-1)+1+\frac{k-3}{2}$.
But $\deg_G(z_1)=k-1$, so it follows that $|A|\leq\frac{k+1}{2}$.
As $A\cup B= V(D)$, this implies that $|B|\geq\frac{k-7}{2}$.
\renewcommand{\qedsymbol}{$\blacksquare$}
\end{proof}

%%%%%%%%%%%%%%%%%%%%%%%%%%%%%%%%%%%%%%%
%%%%%%%%%%%%%%%%%%%%%%%%%%%%%%%%%%%%%%%
\section{Discharging}\label{sec:Discharging}

We start by analyzing the local structure of a minimal counterexample to Theorem \ref{thm:Main1}. Then we complete the discharging argument in two stages; in the first stage we send charge along edges according to established rules, and in the second stage we average charge across the graph.
We define a charge function $w:V(G)\rightarrow\mathbb{R}$ so that for all $v\in V(G)$
\[w(v):=(k-2)(k+1)+\varepsilon-\deg_G(v)(k-1).\]
Note that the total initial charge across $G$ is $\rho(G)+\delta T(G)$, and that the charge of a vertex $x$ with degree $d$ is $w(x)=(k-d)(k-1)-2+\varepsilon$.

We now define the four sets we need to address in the second stage of discharging. 
\[L:=\{v\in V(G) \mid v \text{ is a lone-vertex in a cluster of size 1}\},\]
\[M:=\{v\in V(G) \mid v \text{ is a lone-vertex in a cluster of size 2}\},\]
\[P:=\{v\in V(G)\mid\deg_G(v)=k\},\]
\[Q:=\{v\in V(G)\mid\deg_G(v)=k+1\}.\]
Let $R$ be the set $V(G)-(L\cup M \cup P \cup Q)$ which contains the remaining vertices of $G$.

\textbf{Discharging Rule \#1 (R1)} Every vertex of degree at least $k+2$ reserves charge of $-2+\varepsilon$ and sends the remaining charge equally to all neighbors. 

\textbf{Discharging Rule \#2 (R2):} Every structure-vertex sends total charge $-(k-1)$ spread equally among all neighbors that are near-vertices.

For each vertex $v$, define $w'(v)$ to be the charge after applying (R1) and (R2) to $G$.
Note that a vertex of degree $d$ which follows (R1) sends out charge $(\frac{k}{d}-1)(k-1)$ to each of its neighbors. Also note that if a structure-vertex $x$ sends charge to a near-vertex $y$, then $|C_x|>|C_y|$. 

\begin{lemma}\label{lem:Discharge1}
Apply (R1) and (R2) to a minimal counterexample to Theorem \ref{thm:Main1} $G$ with charge function $w$ as above. For every vertex $v\in V(G)-(L\cup M\cup P\cup Q)$, the new charge $w'(v)$ is at most $-2+\varepsilon$.
\end{lemma}
\begin{proof}
Let $G$ be a minimal counterexample to Theorem \ref{thm:Main1} with charge function $w:V(G)\rightarrow\mathbb{R}$ as above, and apply (R1) and (R2).
If $v$ is a vertex with $\deg_G(v)\geq k+2$, then by (R1) it follows that $w'(v)\leq-2+\varepsilon$.
The cases that we need to check are when $v$ has degree $k-1$ and is either a structure-vertex, near-vertex, or lone-vertex in a cluster of size at least 3.\renewcommand{\qedsymbol}{}\end{proof}

\begin{proof}[Case 1a] Suppose that $v$ is a structure-vertex that is a key vertex of a gadget of $G$.

By Lemma \ref{lem:KeyBound}, the vertex $v$ has at least $\frac{k-3}{2}$ neighbors of degree at least $\frac{3(k-3)}{2}$; we will call these \emph{high-degree} neighbors.
For $k\geq27$, high-degree neighbors have degree at least $\frac{4}{3}k$. Therefore $v$ receives charge of $\frac{-1}{4}(k-1)$ or less from each high-degree neighbor by (R1). 
The vertex $v$ possibly sends charge $-(k-1)$ by (R2) as well. Therefore it follows that
\[w'(v)\leq k-3+\varepsilon-\left(\frac{k-1}{4}\right)\left(\frac{k-3}{2}\right)+(k-1)=\frac{-1}{8}(k-1)(k-19)-2+\varepsilon.\]
Because $k\geq 19$, we have $w'(v)\leq -2+\varepsilon$ as desired.
\renewcommand{\qedsymbol}{}\end{proof}

\begin{proof}[Case 1b] Suppose that $v$ is a structure-vertex that is in a $K_{k-3}$ subgraph of $G$.

By Lemma \ref{lem:Kk3Bound}, the vertex $v$ has at least $\frac{k-9}{6}$ neighbors of degree at least $\frac{3(k-3)}{2}-1$; we will call these \emph{high-degree} neighbors.
For $k\geq 33$, high-degree neighbors have degree at least $\frac{4}{3}k$.
As long as $v$ is not affected by (R2) we have
\[w'(v)\leq k-3+\varepsilon-\left(\frac{k-1}{4}\right)\left(\frac{k-9}{6}\right)=\frac{-1}{24}(k^2-34k+33)-2+\varepsilon3\]
Because $k\geq 33$, we have $w'(v)\leq-2+\varepsilon$ as desired.

If $v$ is affected by (R2), then $v$ has a neighbor outside of the $K_{k-3}$ which is in a different cluster, and by Lemma \ref{lem:Kk3Bound} there are at least $\frac{k-7}{2}$ high-degree neighbors.
In this case we have
\[w'(v)\leq k-3+\varepsilon-\left(\frac{k-1}{4}\right)\left(\frac{k-7}{2}\right)+(k-1)=\frac{-1}{8}(k^2-24k+23)-2+\varepsilon.\]
Because $k\geq23$, we have $w'(v)\leq-2+\varepsilon$ as desired.
\renewcommand{\qedsymbol}{}\end{proof}

\begin{proof}[Case 2] Suppose that $v$ is a near-vertex.

Let $v$ be in a cluster $C_v$ of size $t$ and let $u$ be an adjacent structure-vertex in a cluster $C_u$ of size $s$. 
By (R2) each vertex of $C_v$, including $v$, receives a charge of $\frac{-(k-1)}{r}$ from each vertex of $C_u$. Because $s>r$, the final charge on $v$ is 
\[w'(v)\leq k-3+\varepsilon-\frac{s(k-1)}{r}<-2+\varepsilon.\]
\renewcommand{\qedsymbol}{}\end{proof}

\begin{proof}[Case 3] Suppose that $v$ is a lone-vertex in a cluster $C_v$ of size $r$, where $r\geq3$.

By definition of lone-vertex, $v$ does not have any neighbors $y$ in a cluster $C_y$ with $C_y\neq C_v$. 
Let $y_1,y_2,\ldots y_{k-r}$ be the neighbors of $v$ in $V(G)-C_v$. By Corollary \ref{cor:CloneName} no $y_i$ has degree less than $k-1+r$, as this would imply that $v$ is a structure-vertex.
Therefore, by (R1), each $y_i$ sends charge at most $\left(\frac{k}{k-1+r}-1\right)(k-1)$ to $v$.
It follows that the upper bound on $w'(v)$ is 
\[\hat{w}'(v)\leq k-3+\varepsilon+\left(\frac{1-r}{k-1+r}\right)(k-1)(k-r).\]
The second derivative of $\hat{w}'(v)$ with respect to $r$ is positive for all $k>1$, so we only need to check that $\hat{w}'(v):=-2+\varepsilon$ for $r=3$ and $r=k-4$.
For $r=3$ we have
\[\hat{w}'(v)= k-3+\varepsilon+\frac{-2(k-1)(k-3)}{k+2}=-2+\varepsilon+\frac{(k-1)(8-k)}{k+2}\]
and for $r=k-4$ we have
\[\hat{w}'(v)= k-3+\varepsilon+\frac{(5-k)(k-1)4}{2k-5}=-2+\varepsilon+\frac{(k-1)(15-2k)}{2k-5}.\]
Because $k\geq8$, we get $w'(v)\leq-2+\varepsilon$ as desired.
\renewcommand{\qedsymbol}{$\blacksquare$}
\end{proof}

Lemma \ref{lem:Discharge1}, specifically Case 1b, restricts our main result to $k\geq 33$. Although there is approximation in the proof of this case, using a computer algebra system one can check that the result only holds for $k\geq33$; we paid no penalty in strength of argument by using simplified calculations.
Now that we have verified the charge for vertices in $R$, we need to examine the charge on $L,M,P,Q$ to calculate the total charge.  This gives us a lower bound on the combined size of $L$ and $P$.

\begin{lemma}\label{lem:Discharge2}
In a minimal counterexample $G$ to Theorem \ref{thm:Main1}, let $L$ be the set of lone-vertices in clusters of size 1 and let $P$ be the set of vertices of degree $k$. Then $|L|+|P|>|V(G)|\left(1-\frac{\varepsilon}{2}\right)$.
\end{lemma}
\begin{proof}
For each $x\in L\cup M$, every vertex in $N_G(x)\cap R$ has degree at least $k+2$, so a charge of $-\frac{2(k-1)}{k+2}$ or less is sent along each edge from $R$ to $L\cup M$. Let $e(L\cup M, R)$ denote the number of such edges.
The total charge on $G$ is bounded by
\begin{align}\label{eq:TotalCharge}
\sum_{v\in V(G)} w(v)=\sum_{v\in V(G)}w'(v)&\leq\varepsilon|V(G)|-2|R|+(k-3)(|L|+|M|)-2|P| \nonumber \\
&\indent -(k+1)|Q|-\frac{2(k-1)}{k+2}e(L\cup M,R).
\end{align}
Note that no vertex in $M$ is adjacent to any vertex in $P$ by Corollary \ref{cor:CloneName}.
Thus it follows that 
\begin{equation}\label{eq:LMR}
e(L\cup M, R)=(k-1)|L|-e(L,P\cup Q)+(k-2)|M|-e(M,Q).
\end{equation}
It is clear that $e(L,P\cup Q)\leq (k-1)|L|$, and Lemma \ref{lem:Kernel} shows that $e(L,P\cup Q)\leq 2|L|+2|P|+3|Q|$ (calculations are simpler when we increase the contribution of the independent set $L$).
Then it follows that
\begin{equation}\label{eq:LPQ}
e(L,P\cup Q)\leq\frac{k-4}{2(k-1)}(k-1)|L|+\left(1-\frac{k-4}{2(k-1)}\right)\left(2|L|+2|P|+3|Q|\right).
\end{equation}

Using Lemma \ref{lem:Kernel} we can also bound $e(M,Q)\leq |M|+3|Q|$ . 
With this, Equation \ref{eq:LMR}, and Equation \ref{eq:LPQ} we can rewrite Equation \ref{eq:TotalCharge} as
\begin{align}
\sum_{v\in V(G)}w(v)&\leq \varepsilon|V(G)|-2|R|+(k-3)(|L|+|M|)-2|P|-(k+1)|Q| \nonumber \\
&\indent -(k-3)|L|+2|P|-\frac{2(k^2-4k+3)}{k+2}|M|+\frac{9k}{k+2}|Q| \nonumber \\
&=\varepsilon |V(G)|-2|R|-\frac{k^2-7k+12}{k+2}|M|-\frac{k^2-6k-2}{k+2}|Q|.
\end{align}
As the coefficients of $|M|$ and $|Q|$ are both at most -2 for $k>8$, we have 
\begin{equation*}
\varepsilon|V(G)|-2(|V(G)|-|L|-|P|)\geq\sum_{v\in V(G)}w(v)=\rho(G)+\delta T(G)>0.
\end{equation*}
From this, it follows that $|L|+|P|>|V(G)|\left(1-\frac{\varepsilon}{2}\right)$.
\renewcommand{\qedsymbol}{$\blacksquare$}
\end{proof}

We are now prepared to prove Theorem \ref{thm:Main1}.

\begin{proof}[Proof of Theorem \ref{thm:Main1}]
We first get a bound on the set $L$.
It is clear that $2|E(G)|\geq k|P|+(k-1)|L|$, so by Lemma \ref{lem:Discharge2} it follows that
\begin{equation}\label{eq:EdgeBound}
2|E(G)|> k|V(G)|\left(1-\frac{\varepsilon}{2}\right)-|L|.
\end{equation}

By assumption $\rho(G)>0$, so $2|E(G)|<\left(k+\frac{\varepsilon-2}{k-1}\right)|V(G)|$ by the definition of $\varepsilon$-potential.
Combining this with Equation \ref{eq:EdgeBound}, we have
\begin{equation}
|L|>\frac{|V(G)|}{k-1}\left(2-\varepsilon-\frac{\varepsilon(k^2-k)}{2}\right). \nonumber
\end{equation}

Recall that $\text{mic}(G)$ is the maximum of $\sum_{v\in I}\deg_G(v)$ over all independent vertex subsets $I$, so $\text{mic}(G)\geq(k-1)|L|$.
Kierstead and Rabern (Theorem 2.4 in \cite{KiersteadRabern})
show that $2|E(G)|>(k-2)|V(G)|+\text{mic}(G)$.
Therefore we can improve Equation \ref{eq:EdgeBound} to
\begin{equation}
2|E(G)|>|V(G)|\left((k-2)+2-\varepsilon-\frac{\varepsilon(k^2-k)}{2}\right). \nonumber
\end{equation}

Again, because $\rho(G)>0$ the definition of $\varepsilon$-potential shows that

\[\left(k+\frac{\varepsilon-2}{k-1}\right)|V(G)|>|V(G)|\left((k-2)+2-\varepsilon-\frac{\varepsilon(k^2-k)}{2}\right)\]
and hence
\[\frac{\varepsilon-2}{k-1}>-\varepsilon-\frac{\varepsilon(k^2-k)}{2}.\]

This is equivalent to 
\[\frac{4}{k^3-2k^2+3k}<\varepsilon,\]
which is a contradiction to our choice of $\varepsilon$.
\renewcommand{\qedsymbol}{$\blacksquare$}
\end{proof}

\bibliographystyle{amsplain}
\bibliography{StructureBib}

\end{document}